\documentclass[12pt]{amsart}

\usepackage{mathtools}
\usepackage{amsmath}
\usepackage{graphicx}
\usepackage{subfig}
\usepackage{tabularx}

\usepackage{mathrsfs}     
\usepackage{helvet}         
\usepackage{courier}        
\usepackage{type1cm}      
\usepackage{color}
\usepackage{todonotes}
\usepackage{mathabx,mathbbol}
\usepackage{natbib}
\usepackage{hyperref}
\usepackage{cleveref}

\usepackage{geometry,calc,color}
\usepackage{amsmath,amssymb}

\usepackage{enumerate}
\usepackage{esint}

\usepackage{booktabs}
\usepackage{multirow}

\renewcommand{\theta}{\vartheta}
\renewcommand{\phi}{\varphi}
\renewcommand{\O}{\Omega}

\newcommand{\Tstarh}{\widehat T_h}

\newcommand{\dt}{\partial_t}
\newcommand{\p}{\varphi}
\newcommand{\n}{\mathbf{n}}

\newcommand{\dert}{\partial_t}

\newcommand{\Ku}{ {K_u}}
\newcommand{\Kp}{{K_\p}}
\newcommand{\gRu}{g^R_u}
\newcommand{\gRp}{g^R_{\p}}

\newcommand{\Au}{\mathcal{A}_u}
\newcommand{\Ap}{\mathcal{A}_{\p}}
\newcommand{\C}{\mathcal{C}}
\newcommand{\F}{\mathcal{F}}

\newcommand{\Vh}{\mathbb{V}_h}
\newcommand{\Sh}{\mathbb{S}_h}

\newcommand{\PiLdue}{{\mathbb{\Pi}^0}}
\newcommand{\Cinftyu}{C^\infty_u}

\newcommand{\vareps}{\varepsilon}

\newcommand{\bs}[1]{\boldsymbol{#1}}

\newcommand{\fp}{\bs{\psi}}
\newcommand{\fu}{{\bs{j}}}
\newcommand{\hfp}{\widehat{\bs{\psi}}}
\newcommand{\hfu}{\widehat{\bs{j}}}
\newcommand{\ft}{\bs{\tau}}
\newcommand{\fv}{\bs{q}}
\newcommand{\fw}{\bs{h}}
\newcommand{\hu}{\widehat u}
\newcommand{\hp}{\widehat {\p}}
\renewcommand{\t}{\theta}
\newcommand{\ct}{\widehat\theta}
\newcommand{\hw}{\widehat w}
\newcommand{\GN}{\Gamma_N}
\newcommand{\GR}{\Gamma_R}

\newcommand{\tu}{{\gamma_u}}
\newcommand{\tp}{{\gamma_\phi}}

\newcommand{\Th}{{\mathcal{T}_h}}
\newcommand{\Eh}{{{\mathcal{E}_h}}}
\newcommand{\EhK}{{{\mathcal{E}^K_h}}}

\newcommand{\Hdiv}[1]{H(\mathrm{div},#1)}
\newcommand{\Poly}[1]{\mathbb{P}_{#1}}
\newcommand{\PiHDG}{\mathbb{\Pi}^\mathrm{dg}}

\newcommand{\ve}{v^e}
\newcommand{\fve}{\boldsymbol{q}^e}
\newcommand{\hve}{\widehat{v}^e}
\newcommand{\Ve}{V^e}

\newcommand{\vpi}{v^\pi}
\newcommand{\fvpi}{\boldsymbol{q}^\pi}
\newcommand{\hvpi}{\widehat{v}^\pi}
\newcommand{\Vpi}{V^\pi}

\newcommand{\imu}{\alpha}
\newcommand{\inu}{\beta}

\newcommand{\bK}{{\partial K}}

\newcommand{\V}{\mathbb{V}}

\newcommand{\he}{\widehat e}

\newcommand{\uI}{u^I_h}

\newcommand{\fuI}{\fu_h^I}
\newcommand{\huI}{\hu_h^I}
\newcommand{\fuIh}{\fu_h^I}
\newcommand{\huIh}{\hu_h^I}

\newcommand{\fe}{\boldsymbol d}

\newcommand{\heta}{\widehat \delta}
\newcommand{\feta}{\boldsymbol \eta}

\newcommand{\hzeta}{\widehat \xi}
\newcommand{\fzeta}{\boldsymbol \zeta}

\newcommand{\Zeta}{\Xi}

\renewcommand{\he}{\widehat e}
\newcommand{\myeta}{\delta}

\newcommand{\Tstarhn}{\widehat T_{h_n}}

\newcommand{\uIh}{{\uI}}
\newcommand{\UIh}{{ U^I_h}}

\newcommand{\HDGA}{$\mathrm{HDG}(k-1,k)$}
\newcommand{\HDGk}{$\mathrm{HDG}(k,k)$}

\newcommand{\HDGAzu}{$\text{HDG}(0,1)$}
\newcommand{\HDGAud}{$\text{HDG}(1,2)$}
\newcommand{\HDGku}{$\text{HDG}(1,1)$}
\newcommand{\HDGkd}{$\text{HDG}(2,2)$}

\definecolor{NewText}{RGB}{140,0,90}

\newcommand{\epsB}{\varepsilon_B}
\newcommand{\sbar}{{\widetilde s}}
\newcommand{\sigbar}{{\widetilde{\jsmooth}}}
\newcommand{\rbar}{{\widetilde{\psmooth}}}
\newcommand{\Rrate}{\mathcal R_{\kj,k}}
\newcommand{\Srate}{\mathcal S_{\kj,k}}

\newcommand{\gNu}{g^N_u}
\newcommand{\gNp}{g^N_\p}

\newcommand{\ConstPsi}{K_{\fp}}

\newcommand{\PiO}[1][\Th]{\Pi^0_{#1}}               
\newcommand{\fPiO}[1][\Th]{\boldsymbol\Pi^0_{#1}}   
\newcommand{\PiE}[1]{\Pi^\partial_{#1}}              
\newcommand{\PiHs}[1]{\Pi^{#1}}                      
\newcommand{\PiHv}[1]{\boldsymbol\Pi^{#1}}           

\newcommand{\hs}{\widehat s }

\renewcommand{\PiHDG}{\Pi_{\mathrm{HDG}}}

\newcommand{\Cpsi}{c_{\fp}}

\newcommand{\Tstar}{T'}
\newcommand{\HTh}[1]{{#1}}
\newtheorem{theorem}{Theorem}[section]
\newtheorem{corollary}[theorem]{Corollary}
\newtheorem{lemma}[theorem]{Lemma}
\newtheorem{proposition}[theorem]{Proposition}
\theoremstyle{definition}

\newtheorem{problem}[theorem]{Problem}
\newtheorem{assumption}[theorem]{Assumption}

\theoremstyle{remark}
\newtheorem{remark}[theorem]{Remark}

\numberwithin{equation}{section}

\providecommand{\EHB}[1]{#1} 
\providecommand{\PiO}[1][\Th]{\Pi^0_{#1}}\providecommand{\fPiO}[1][\Th]{\boldsymbol\Pi^0_{#1}}
\providecommand{\PiE}[1]{\Pi^\partial_{#1}}\providecommand{\PiHs}[1]{\Pi^{#1}}\providecommand{\PiHv}[1]{\boldsymbol\Pi^{#1}}%
\providecommand{\HTh}[1]{#1}

\newcommand{\fpiu}{\fu^\pi}
\newcommand{\upi}{u^\pi}
\newcommand{\hupi}{\hu^\pi}

\newcommand{\ctsens}{\widetilde\chi\,}
\newcommand{\jsmooth}{\sigma}
\newcommand{\psmooth}{r}

\begin{document}
	
	\allowdisplaybreaks
	
	\addtolength{\jot}{1.5mm}
	\abovedisplayskip=7pt
	\belowdisplayskip=7pt
	
	\title[HDG for the Keller--Segel model]{Hybridizable Discontinuous Galerkin methods for the Keller--Segel model} %

\author{S. Bertoluzza}
\address[Silvia Bertoluzza]{Istituto di Matematica Applicata e Tecnologie Informatiche "E. Magenes", CNR, Pavia, Italy}\email{silvia.bertoluzza@imati.cnr.it}

\author{G. Bretti}
\address[Gabriella Bretti]{Istituto per le Applicazioni del Calcolo "M. Picone", CNR, Rome, Italy}
\email{g.bretti@iac.cnr.it}

\author{M. Pennacchio}
\address[Micol Pennacchio]{Istituto di Matematica Applicata e Tecnologie Informatiche, "E. Magenes", CNR, Pavia, Italy}\email{micol.pennacchio@imati.cnr.it}

\author{C. Prud'homme}
\address[Christophe Prud’homme]{Universit\'e  de Strasbourg, CNRS, IRMA UMR 7501, Cemosis, Strasbourg, France}
\email{christophe.prudhomme@cemosis.fr}

	

	
	\thanks{This work was supported by the Italian Ministry of Research, under the complementary actions to the NRRP ``D34Health - Digital Driven Diagnostics, prognostics and therapeutics for sustainable Health care'' Grant \#PNC0000001,  CUP B53C22006100001. As part of the “France 2030” initiative, this work has benefited from a national grant managed by the French National Research Agency (Agence Nationale de la Recherche), attributed to the ExaMA project of the NumPEx PEPR program under reference ANR-22-EXNU-0002.
		MP has been partially supported by ICSC-Centro Nazionale di Ricerca in High Performance Computing, Big Data, and Quantum Computing funded by the European Union--NextGenerationEU.
		SB, GB and MP are members of the Istituto Nazionale di Alta Matematica - Gruppo Nazionale di Calcolo Scientifico.}

\begin{abstract}
%
	{We propose and analyse two hybridizable discontinuous Galerkin (HDG) methods for the time-dependent Keller–Segel model, a drift--diffusion system describing chemotactic cell movement.
		The HDG methods considered are based on a mixed formulation in which the total cellular flux, including the chemotactic drift, is an independent unknown. When this flux is approximated with the same polynomial degree $k$ as the other unknowns, we prove convergence of order $k+1$ in $L^2$ for all the variables, fluxes included. When it is approximated with degree $k-1$, only the cellular flux converges with the lower order $k$. Numerical experiments confirm the theoretical rates. }
	\end{abstract}

	\maketitle

	
	\section{Introduction}

	Chemotaxis, i.e. directed cell migration in response to chemical signals coming from external sources or secreted by the organisms themselves, plays a key role in a number of biological processes, from bacterial movement, to inflammation, to cancer development, to embryology. 
	The taxis can be either positive (attractive) or negative (repulsive) due to the presence of cells/chemicals or nutrient depletion.
	The dynamics of many amoeboid cells, including {\em Physarum polycephalum}, {\em Dictyostelium discoideum} or leukocytes and some eukaryotic cells, shows that chemotactic movement replaces random and undirected exploration of the environment in the presence of chemical gradients: amoeba gathering to aggregate around food sources \cite{nature, WILLARD}, bacteria alternating phases of free swimming and tumbling in order to reorientate according to nutrient or even toxic external substances \cite{berg,bacteria}.
	The physiological mechanism in human body's immune response to wounds or infection is regulated by chemotaxis as well. Indeed, chemokines in the inflammatory phase are responsible for the recruitment of macrophages and the promotion of angiogenesis \cite{mantovani,wound}.
	In the brain, chemotaxis plays a fundamental role in the neural migration for cortical development \cite{brain}. The process of demyelination is regulated by chemotaxis, see \cite{CALVEZ} and references therein.
	Regarding cancer, one of the most interesting examples of chemotaxis involves immune cells that, besides other functions crucial for the organism's defense, have the role of detecting cancer cells and fighting them \cite{roussos}. The immunosurveillance in cancer inflammation involves complex relationships between cells and chemotactic factors secreted in the tumor microenvironment, as well as the interactions among cells, see \cite{roselli} and references therein.
	
	The present paper focuses on the classical Keller--Segel (KS) model \cite{KS}, proposed in 1970 and still considered as a fundamental model, thanks  to its basic formulation of biological processes. In spite of this apparent simplicity, such a model is able to replicate key behavior of chemosensitive populations, such as directed movement of individuals and pattern formation. The evolution of the density of individuals is described by a parabolic equation, and the concentration of chemoattractant is generally given by an elliptic/parabolic equation. Since in this model the organisms secrete chemical substances and move towards areas of higher chemical concentration, this leads to high aggregation of individuals, expressed mathematically as a blow up of their density. The high nonlinearity due to the chemotactic term makes the KS system largely intractable by analytical methods. Moreover, the model may develop two kinds of solutions: unbounded in finite time (blow up) or bounded with peaks in small areas of the domain \cite{sleeman}. Many analytical studies on blow-up in KS equations are described in \cite{dejak} and references therein. 
	Here, the behavior of solutions of the KS model is studied and analyzed in the bounded regime since the biological application we have in mind has a slightly different formulation where blow-up does not occur, see \cite{BBN2021}.
	Numerical tests in the blow-up regime are also provided in order to show the capability of the method to handle such a challenging case.

	The KS model is hard to solve numerically due to the singular structure of solutions and to high sensitivity to changes in model parameters or in the step sizes of the numerical scheme. 
	For the reasons described above, it is fundamental to have
	 efficient numerical methods at our disposal, able to produce a realistic approximation of the observed
	phenomenon in order to describe correctly the evolution of concentration profiles. In particular, a numerical scheme for KS needs to reproduce all the relevant physical features, like the nonnegativity of the discrete approximation of
	the solution as well as the conservation of total mass in time in the case of no-flux boundary conditions. Moreover, since the density may grow fast in very small regions as a result of directed movement, the numerical approximation must reproduce this behavior and avoid non-physical oscillations and numerical instabilities.
	
	Different numerical methods have been proposed for tackling such a system of
	equations; see, e.g.,
	\cite{Bonilla,epshteyn_discontinuous_2009,EpshteynKurganov2008,Filbet2006,Kurganov,HASSAN}.
  In \cite{epshteyn_discontinuous_2009} the discontinuous Galerkin method is shown to be well suited to tackle the numerical simulation of such a mechanism, and a local discontinuous Galerkin method 
was proposed in \cite{LiShuYang2017}, where positivity is enforced by means of the limiter of \cite{zhangShu2010}. 
	The efficiency of discontinuous Galerkin methods is however hindered by the proliferation of degrees of freedom.
	The hybridizable discontinuous Galerkin (HDG) method was developed with the goal of overcoming such a drawback \cite{cockburn_unified_2009}. We devote this paper to the analysis of such a method when applied to the Keller--Segel equation. 
	More precisely, we consider the following system of equations:
	\begin{gather}\label{KS1}
		\dert u - \nabla\cdot(\mu \nabla u - \chi u \nabla\p) = f_u,\\ \label{KS2}
		\dert \p - \nabla\cdot \nu \nabla \p + a \p - b u = f_{\p}.
	\end{gather}
In the above equations, $\mu, \nu, a, b$ and $\chi$ are given constants, with
$\mu, \nu, a$ and $b$ positive, while $\chi$ can have either sign ($\chi>0$
corresponding to chemoattraction and $\chi<0$ to chemorepulsion).

%
	
	The HDG formulation that we consider to discretize \eqref{KS1}--\eqref{KS2} is related to the
	one proposed in \cite{chen2019hdg} for drift--diffusion models of semiconductor
	devices, and our error analysis follows the same general strategy. There are,
	however, three main differences.
	
	The first one lies in the choice of the flux variable. In
	\cite{chen2019hdg} the diffusive flux is taken as an independent unknown, while
	the drift term is kept in the conservation equation. Here the unknown is instead
	the total cellular flux $\boldsymbol j = -\mu\nabla u + \chi u\nabla\varphi$, as in
	the HDG method of \cite{CockburnDong2009} for linear steady-state
	convection--diffusion--reaction problems, so
	that the conservation equation reduces to
	$\partial_t u + \nabla\cdot\boldsymbol j = f_u$, and the nonlinearity only enters
	the algebraic relation defining $\boldsymbol j$.
	This allows us, besides
	approximating the flux with polynomials of degree $k-1$, as in
	\cite{chen2019hdg},
	to approximate it also with the same degree $k$ used for all the other
	unknowns. For the resulting equal-order method we prove convergence of order
	$k+1$ for all the variables, including the cellular flux.
	%
%
In the linear steady-state case, this rate was observed numerically in
\cite{CockburnDong2009}, and an analysis based on the projection of
\cite{cockburn2010projection} was sketched in \cite[Section~5]{ChenCockburn2012}.
Here we provide a complete proof for the nonlinear, time-dependent system
\eqref{KS1}--\eqref{KS2}. Compared with the case of degree $k-1$, the proof is also
simpler: the key bound on the projection of the cellular variables follows
from an energy argument, and requires neither a duality argument nor the
convexity of the domain.

	The second difference is that the chemoattractant satisfies a parabolic equation,
	rather than the elliptic equation satisfied by the potential in the semiconductor
	model,
	so that its error has to be controlled in time together with the one on the cell
	density.
	
	The third difference concerns the boundary conditions which, in view of the
	application to the simulation of cancer-on-chip devices, we take of mixed
	Neumann--Robin type, in order to handle the membrane type condition that models
	the junction between channels and chambers. When the flux is approximated with
	degree $k-1$, the optimal bound on the cell density relies on a duality argument,
	and the reduced regularity of the adjoint problem under mixed boundary conditions
	results in the loss of an arbitrarily small power $h^{\varepsilon}$.
	
	Summarizing, when the cellular flux is approximated with degree $k-1$ we prove
	convergence of order $k+1$ (or $k+1-\varepsilon$, in the mixed Neumann--Robin
	case) for all the variables except the cellular flux, for which we prove
	convergence of order $k$. When it is approximated with degree $k$, we prove
	convergence of order $k+1$ for all the variables.

	The paper is organized as follows. In Section \ref{sec:mixed}, we derive the mixed formulation of system \eqref{KS1}--\eqref{KS2} and introduce the two additional flux unknowns $\fu$ and $\fp$. This serves as a preliminary step for the hybridized discontinuous Galerkin discretization presented in Section \ref{sec:HDG}.
	In the same section, we state the main convergence result, while the proof is deferred to Section \ref{sec:proof}, with additional technical details collected in the Appendix. 
%
	Numerical results validating the theory are reported in
	Section~\ref{sec:numerical}, and some conclusions are drawn in
	Section~\ref{sec:conclusions}.

	\section{Mixed form of the Keller--Segel equations}\label{sec:mixed}
Let $\Omega \subset \mathbb{R}^2$ denote a convex polygonal domain.
	In the following we will consider the  system \eqref{KS1}--\eqref{KS2} in a time interval $[0,T]$, with $T > 0$, endowed  with initial conditions 
	\begin{equation}
		\label{initials} u(0,\cdot) = u_0, \qquad \p(0,\cdot) = \p_0 \qquad \text{ in }\O, 
	\end{equation}
	and with mixed Neumann--Robin conditions: 
	\begin{gather}
		\nu \nabla \p \cdot \n = {\gNp}, \qquad  \qquad (\mu\nabla u - \chi u \nabla \p)\cdot \n = {\gNu}, \qquad  \text{ on }\GN,\label{NeumannBCstrong}\\
		\nu \nabla \p \cdot \n + \frac 1 {\varepsilon_\p} \p = \gRp, \qquad (\mu\nabla u - \chi u \nabla \p)\cdot \n  + \frac 1 {\varepsilon_u} u = \gRu,  \qquad  \text{ on }\GR,\label{RobinBCstrong}
	\end{gather}
	where $\GN$ and $\GR$  respectively denote the Neumann and Robin portions of the boundary. We assume that $\GN$ and $\GR$ satisfy $\GN\cap \GR=\emptyset$ and $\partial\Omega=\overline\Gamma_N\cup\overline\Gamma_R$, and that they are either empty, or unions of whole edges of the polygon (we allow for aligned edges, i.e. edges forming a $\pi$ angle, to account for switches of boundary condition type within a physical edge).  Here  $\varepsilon_\p$ and $\varepsilon_u$ are positive constants. {The boundary data $\gNu$, $\gNp$ and $\gRu$, $\gRp$ are given functions on $\GN$ and on $\GR$, respectively; for $\gNu=\gNp=0$ one recovers the usual no-flux condition on $\GN$.}
	
	\

In order to write down the hybridized discontinuous Galerkin discretization of  \eqref{KS1}--\eqref{KS2}, we write the system in mixed form.
To this aim we  introduce the two new unknowns $\fu$ and $\fp$ defined by the equations
\begin{gather}
	\label{fluxp}	\fp + \nu  \nabla \p = 0,\\
	\label{fluxu}	\fu + \mu \nabla u - \chi   u \nabla \p = 0.
\end{gather}
We let
\[\alpha = \mu^{-1}, \qquad \beta = \nu^{-1} \qquad \text { and }\qquad \ctsens = \chi/\mu,\] whence
equations  \eqref{KS1}--\eqref{KS2} and \eqref{fluxp}--\eqref{fluxu} can be rewritten as 
\begin{gather}
	\partial_t  u + \nabla \cdot \fu = f_u, \label{KS1_s}\\
	\partial_t  \phi + \nabla \cdot \fp + a \p - b u= f_\phi, \label{KS2_s}\\
	\inu \fp +  \nabla \p = 0, \label{fluxp_s}\\
	\imu \fu + \nabla u +  \inu \ctsens  u \fp = 0. \label{fluxu_s}  
\end{gather}
The Neumann and Robin boundary conditions respectively rewrite as
\begin{gather*}
	{- \fp \cdot \n = \gNp}, \quad \text{ on }   \GN, \qquad {- \fu \cdot \n = \gNu} \quad \text{ on } \GN, \\
-	\fp \cdot \n + \frac 1 {\varepsilon_\p} \p = \gRp, \qquad - \fu\cdot \n  + \frac 1 {\varepsilon_u} u = \gRu,  \qquad  \text{ on }\GR.
\end{gather*}

\begin{remark}
	While for the sake of notational simplicity we did not consider the situation where on a part of the boundary the two unknowns $u$ and $\phi$ satisfy boundary conditions of different type, the 
	method that we will consider  adapts to such a situation in a straightforward manner.
\end{remark}

\section{Hybridized discontinuous Galerkin discretization}\label{sec:HDG}

\subsection{Notation}
We let $\Th$ denote a quasi-uniform shape-regular triangulation of the domain $\Omega$, with mesh size $h$, and we let $\Eh$ denote both the skeleton $\Eh = \cup_K \partial K$ and the set of edges of $\Th$ (which of the two definitions applies will be clear from the context).  Moreover, for $K\in \Th$ we let $\EhK \subset \Eh$ denote the set of edges of $K$. For $u \in L^2(\O)$ and $\fu \in (L^2(\O))^2$ we let $u^K = u|_K$ and $\fu^K = \fu|_K$. 

We let
\begin{gather}
	\HTh{H^1(\Th)} = \{u \in L^2(\O)\ \text{such that}\ \nabla u^K \in (L^2(K))^2\ \text{ for all } K \in \Th \},\\
	\Hdiv{\Th} = \{\fu \in( L^2(\O))^2 \ \text{such that}\ \nabla \cdot \fu^K \in L^2(K)\ \text{ for all } K \in \Th \}.
\end{gather} 
Using the notation $(\,\cdot\,)_{K\in\Th}$ to indicate vectors with index in $\Th$, by abuse of notation, for $u \in \HTh{H^1(\Th)}$, $\fu \in \Hdiv{\Th}$ we let
\begin{gather}
	\nabla u = (\nabla u^K)_{K\in\Th} \in \prod_K (L^2(K))^2,\\
	\nabla \cdot \fu = (\nabla \cdot \fu^K)_{K\in\Th}  \in \prod_K L^2(K), \\ \fu\cdot \n = (\fu^K \cdot \n^K)_{K\in\Th}  \in \prod_K H ^{-1/2}(\partial K)\end{gather} 
respectively denote the piecewise gradient of $u$ and the piecewise divergence and the piecewise normal trace of $\fu$.
We will use the following notation for the $L^2$ scalar products
\[
(u,v)_\O = \int_\Omega u v, \qquad (u^K,v^K)_K = \int_K u^K v^K.
\]
We also introduce the following compact notation for  broken scalar products: for $A, B$  in $\prod_K L^2(K)$ and $C, D$ in $\prod_K L^2(\partial K)$, 
\[ 
(A , B)_{\Th} = \sum_{K} \int_K A^K B^K, \qquad (C, D)_{\Eh} = \sum_K \int_{\partial K} C^K D^K
\]
with corresponding broken norms
\[
\| A \|_{\Th}^2 = (A , A)_{\Th}, \qquad \| C \|_{\Eh}^2 = (C , C)_{\Eh}.
\]
Combining classical trace inequalities in each element, we obtain the following standard result\footnote{
We will use the notation $A \lesssim B$ (resp. $A \gtrsim B$) to indicate that the quantity $A$ is less or equal (resp. greater or equal) than a constant times the quantity $B$ with the constant independent of $h$ and $t$, but possibly depending on $k$, on the shape regularity of the mesh, on the different parameters of the model, and on other quantities that we will indicate if need be. $A \simeq B$ will stand for $A \lesssim B \lesssim A$.}:
\begin{equation}\label{trace}
\| w \|_{\Eh}^2 \lesssim h^{-1}\| w  \|^2_\O + h\| \nabla  w \|_{\Th}^2, \qquad
\| \fw\cdot \n \|_{\Eh}^2 \lesssim h^{-1}\| \fw  \|^2_\O + h\| \nabla  \fw \|_{\Th}^2.
\end{equation}
A more general trace inequality also holds: $u \in \HTh{H^s(\Th)}$,  $1/2 < s \le 1$ implies
\begin{equation}\label{trace_s}
	\| u \|^2_{\Eh} \lesssim h^{-1} \| u \|_{\O}^2 + \| u \|_{\O}^{2-1/s} \| u \|_{\HTh{H^{s}(\Th)}}^{1/s}.
\end{equation}

We observe that, with our notation, for all $u \in \HTh{H^1(\Th)}$, $\fu \in \Hdiv{\Th}$ the result of an element by element integration by parts can be written in compact form as
\[
(u,\nabla\cdot \fu)_{\Th} = -(\nabla u,\fu)_{\Th} + (u ,\fu\cdot \n)_{\Eh}.
\]

We will denote by $\Poly{k}(\Th)$ and $\Poly{k}(\Eh)$, respectively, the spaces of discontinuous piecewise polynomials of degree at most $k$  on the mesh $\Th$, and the space of single-valued, discontinuous piecewise polynomials of degree at most $k$ on the skeleton $\Eh$
\begin{gather*}
	\Poly{k}(\Th) = \{ p \in L^2(\O): p|_K \in \Poly{k} \ \text{ for all }K \in \Th\},\\
	\Poly{k}(\Eh) = \{ p \in L^2(\Eh): p|_e \in \Poly{k} \ \text{ for all }e \in \Eh\},\\
		\Poly{k}(\EhK) = \{ p \in L^2(\partial K): p|_e \in \Poly{k} \ \text{ for all }e \in \EhK\}.
\end{gather*}
Finally, we recall that, for all $\fv \in (\Poly{k}(\Th))^2$ we have the following inverse inequality
\begin{equation}\label{inverse}
	h \| \fv \cdot \n \|_{\Eh}^2 \lesssim \| \fv \|^2_{\O}.
\end{equation}

\smallskip

\newcommand{\kj}{{k'}}

\subsection{HDG formulation of the KS equation}\label{sec:HDG_KS}
In order to write down the HDG formulation of the Keller--Segel system, we introduce the traces of $u$ and $\p$ on the skeleton $\Eh$ as additional independent unknowns, which we denote by $\hu$ and $\hp$.

\smallskip

Substantially following \cite{chen2019hdg}, we choose discretizations of order $k \geq 1$ for all the unknowns with the exception of $\fu$  and of the trace $\hu$ of $u$ on $\Eh$, for which we select an order $\kj$ approximation, with $\kj \in \{k-1,k\}$.
As customary in deriving the HDG formulation,
we multiply the flux equations \eqref{fluxu_s}  and \eqref{fluxp_s} 
by test functions $\fv \in (\Poly{\kj}(\Th))^2$ and $\ft \in (\Poly{k}(\Th))^2$ respectively, and, replacing $u$ and $\p$ on $\partial K$ by  $\hu = u|_{\Eh}$ and $\hp=\p|_{\Eh}$, we obtain the identities, 
\begin{equation}\label{physequ}
	0 = (\imu \fu +  \inu \ctsens u \fp ,\fv)_{\Omega} + (\nabla u,\fv)_{\Th}
	=  (\imu \fu + \inu \ctsens u \fp ,\fv)_{\Th}
	- (u ,
	\nabla \cdot \fv)_{\Th} +  (\hu , \fv\cdot \n)_{\Eh}
\end{equation}
and
\begin{equation}\label{physeqp}
	0 = (\inu \fp,\ft)_{\Omega} +  (\nabla \p,\ft)_{\Th}
	=   (\inu \fp,\ft)_{\Th}
	- (\p ,
	\nabla \cdot \ft)_{\Th} + (\hp , \ft\cdot \n)_{\Eh}.
\end{equation}

Now testing equation \eqref{KS2_s} with $\t \in \Poly{k}(\Th)$ yields:
\begin{multline}
	(f_\p,\t)_\Omega = (
	\dert \p + \nabla\cdot \fp + a \p - b u,\t
	)_{\Th} \\
	=  (
	\dert \p,\t)_{\Omega} + (\nabla\cdot  \fp,\t)_\Th + (a \p,\t)_\Omega - (b u,\t)_{\Omega}
	\\
	=	(\dert
	\p,\t)_\Omega  - (  \fp,\nabla\t)_{\Th} + ( \fp\cdot\n, \t)_{\Eh}+ (a \p,\t)_\Omega - (b u,\t)_\O,
\end{multline}
while equation \eqref{KS1_s}, tested with $w \in \Poly{k}(\Th)$, becomes
\begin{gather*}
	(f_u,w)_\Omega =  (\dert u,w)_{\Omega} + 
	(\nabla \cdot \fu ,w)_{\Th} =
	(\dert u,w)_{\Omega} 
	-	(\fu ,\nabla w)_{\Th} + (\fu \cdot \n, w)_{\Eh}.
\end{gather*}
Finally, continuity of the fluxes combined with the boundary conditions implies that 
\begin{gather*}
-  ( \fp \cdot \n,\ct)_{\Eh} + \Kp(\hp,\ct)_{\GR} = (\gRp,\ct)_{\GR} {+ (\gNp,\ct)_{\GN}},\\
 - ( \fu  \cdot \n, \hw)_{\Eh} + \Ku(\hu,\hw)_{\GR} = (\gRu,\hw)_{\GR} {+ (\gNu,\hw)_{\GN}},
\end{gather*}
where 
\[\Kp = \varepsilon_\p ^{-1},\qquad \Ku = \varepsilon_u^{-1},\]
(the minus sign in the flux continuity equations is convenient for analysis, and of course it can be disregarded for implementation).

\newcommand{\HDGU}{HDG($\kj$,$k$)}

\subsection{The HDG scheme}
\label{sec:HDGU-scheme}
Our definition of the cellular flux $\fu$ in \eqref{fluxu} differs from the formulation presented in \cite{chen2019hdg}, where the diffusive flux is introduced separately, while the drift term remains in the conservation equation. 
This formulation allows us to consider two different HDG schemes 
that differ in the polynomial space considered for $\fu$ and in the stabilization for the numerical flux. 
We use the notation $\mathrm{HDG}(\kj,k)$ 
where the first index is the polynomial
degree of the cellular flux $\fu_h$ and of the trace $\hu_h$, whereas the second is the common degree
of $u_h$ and of the whole chemoattractant block $\Phi_h$. 
Thus the
two methods considered in this paper are \HDGA\ and \HDGk.

\

For $\kj\in \{k-1,k\}$ we introduce the following compound discrete spaces:
\begin{equation}\label{discrete_spaces}
	\Vh = \Poly{k}(\Th) \times \Poly{\kj}(\Eh) \times (\Poly{\kj}(\Th))^2, \qquad \Sh = \Poly{k}(\Th) \times \Poly{k}(\Eh) \times (\Poly{k}(\Th))^2.
\end{equation}
Following \cite{chen2019hdg} as well as  \cite{oikawa2015hybridized,lehrenfeld2016high}  we  define numerical fluxes on the skeleton of the mesh $\Th$ as
\begin{equation}\label{numflux}
	\hfu^K \cdot \n^K = \fu^K \cdot \n^K + \tu \, h^{-\kappa} \PiE{\kj} (u^K - \hu), \qquad \hfp^K \cdot \n^K = \fp^K \cdot \n^K + \tp\, (\p^K - \hp),
\end{equation}
where $\kappa = k-\kj$, $\tu$ and $\tp$ are suitable positive constants, and where, for $m \geq 0$, $\PiE{m}: \prod_{K\in \Th} L^2(\partial K) \to  \prod_{K\in \Th} \Poly{m}(\EhK)$ denotes the $L^2$ projection onto piecewise order $m$ polynomials. If $v \in L^2(\Eh)$ is single valued, then $\PiE{m} v \in \Poly{m}(\Eh)$, and $\PiE{m}$ coincides with the $L^2(\Eh)$ projection onto $\Poly{m}(\Eh)$, which, by abuse of notation, we also denote by $\PiE{m}$.

We then consider the following global discrete problem.
\begin{problem}[The \HDGU\ semidiscrete problem]
	\label{problem:HDGA}
	Find $U_h = (u_h, \hu_h, \fu_h) \in \Vh$, $\Phi_h =(\p_h,\hp_h,\fp_h) \in \Sh$
	such that for all $W = (w,\hw,\fw) \in \Vh$, $\Theta = (\t,\ct,\ft) \in \Sh$ it holds that
	\begin{gather}
		\label{KS-HDG2}	\dert (
		\p_h,\t)_\Omega +	( \hfp_h\cdot\n, \t)_{\Eh}  -	(  \fp_h,\nabla\t)_{\Th}   + (a \p_h,\t)_\Omega- (b u_h,\t)_\Omega
		= 	(f_\p,\t)_\Omega, \\ 
		\label{KS-HDG1}
		(\dert u_h,w)_{\Omega} +  (\hfu_h \cdot \n, w)_{\Eh} - 
		(\fu_h ,\nabla w)_{\Th} = 	(f_u,w)_\Omega ,
		\\ 
		\label{KS-HDG4}		  (\inu \fp_h,\ft)_{\Omega}
		- (\p_h ,
		\nabla \cdot \ft)_{\Th} + (\hp_h , \ft\cdot \n)_{\Eh}
		= 0,\\ 
		\label{KS-HDG3}	(\imu \fu_h + \inu \ctsens u_h \fp_h ,\fw)_{\Omega}
		- (u_h ,
		\nabla \cdot \fw)_{\Th} +  (\hu_h , \fw\cdot \n)_{\Eh}
		=0,\\ 
		\label{KS-HDG5}-  ( \hfp_h \cdot \n,\ct)_{\Eh} + \Kp(\hp_h,\ct)_{\GR} = (\gRp,\ct)_{\GR} {+ (\gNp,\ct)_{\GN}}, \\
		\label{KS-HDG6} - ( \hfu_h  \cdot \n, \hw)_{\Eh} + \Ku(\hu_h,\hw)_{\GR} = (\gRu,\hw)_{\GR} {+ (\gNu,\hw)_{\GN}},
	\end{gather}
	and such that the following initial conditions hold
	\begin{equation}
		\label{initials-HDG} (u_h(0,\cdot),p)_\O = (u_0,p)_\O, \qquad
	(\p_h(0,\cdot),q)_\O = (\p_0,q)_\O, 
		\qquad \forall p,q \in \Poly k(\Th).
	\end{equation}
\end{problem}

The following proposition follows from standard arguments (see, e.g., \cite{cockburn_unified_2009} for the unique solvability of HDG local problems, as well as \cite{thomee2006galerkin} for the reduction of semidiscrete Galerkin schemes to systems of ODEs).
\begin{proposition}[Local well-posedness of the semidiscrete problem]\label{prop:local-existence}
	Let $\kj\in\{k-1,k\}$ and $\tu,\tp>0$, and assume that
	\begin{gather*}
	u_0,\p_0\in L^2(\O),\qquad f_u,f_\p\in C^0([0,T];L^2(\O)),\\
	 \gRu,\gRp\in C^0([0,T];L^2(\GR)),\qquad \gNu,\gNp\in C^0([0,T];L^2(\GN)).
	\end{gather*}
	Then Problem~\ref{problem:HDGA} admits a unique maximal solution $(U_h,\Phi_h)$, defined on a maximal interval $J_h\subseteq [0,T]$ and satisfying
	\begin{gather*}
	u_h,\p_h\in C^1(J_h;\Poly{k}(\Th)),\\
	(\hu_h,\fu_h)\in C^0\bigl(J_h;\Poly{\kj}(\Eh)\times(\Poly{\kj}(\Th))^2\bigr),\qquad (\hp_h,\fp_h)\in C^0\bigl(J_h;\Poly{k}(\Eh)\times(\Poly{k}(\Th))^2\bigr),
	\end{gather*}
	 with either $J_h = [0,T]$ or $J_h = [0,T_h)$ for some $0< T_h \leq T$ and 
	\begin{equation}\label{eq:blowup-alternative}
	\limsup_{t\to T_h^-} \|u_h(t)\|_{L^\infty(\O)} = \infty.
	\end{equation}
\end{proposition}

We make the following assumptions on the solution $u, \p$ to \eqref{KS1}--\eqref{KS2}.
\begin{assumption}\label{ass:smoothness}
	We have that
	\begin{gather}
		\label{smoothnessu}		u \in H^1(0,T;H^s(\O)), \qquad 2 < s,\\
		\label{smoothnessj}		\fu \in H^1\bigl(0,T;(H^\jsmooth(\O))^2\bigr), \qquad 1 < \jsmooth,\\
		\label{smoothnessp}		\p \in W^{1,\infty}(0,T;H^\psmooth(\O)), \qquad 2 < \psmooth .
	\end{gather}
	\end{assumption}

	\begin{remark}
	We point out that, provided
		$\fp\in W^{1,\infty}(0,T;W^{s-1,\infty}(\O)^2)$,
		\eqref{smoothnessu} and the definition \eqref{fluxu} imply
		\eqref{smoothnessj} with $\jsmooth=s-1$.  In two space dimensions, by the Sobolev embedding theorem, a
		sufficient condition for the required regularity of $\fp=-\nu\nabla\p$ is
		$\psmooth>s + 1$.
	\end{remark}

\

The following theorem holds.

\begin{theorem}[Convergence of \HDGU]\label{thm:KSerror}
	Let $\kj\in\{k-1,k\}$ and $\kappa=k-\kj$, and set
	\begin{equation}\label{eq:def_capped_indexes}
	\sbar = \min\{s,k+1\}, \qquad \sigbar = \min\{\jsmooth,\kj+1\}, \qquad \rbar = \min\{\psmooth,k+2\}.
	\end{equation}%
	Set
	\[
	\epsB=
	\begin{cases}
		\varepsilon, & \text{if } \kj=k-1 \text{ and both } |\GN|>0 \text{ and } |\GR|>0,\\
		0, & \text{otherwise},
	\end{cases}
	\]
	with $	0<\varepsilon< 1/2$ arbitrary but fixed.
	Assume \eqref{smoothnessu}, \eqref{smoothnessj}, \eqref{smoothnessp}. If $k=1$ and $\kj=0$, assume in addition that $\jsmooth \ge 3/2-{\epsB}$.
	Let $\tu\simeq\tp\simeq1$.  Then there exists $h^*>0$ such that, for all
	$h\leq h^*$,
	\begin{gather}
		\label{main3}\| \p-\p_h  \|_{L^\infty(0,T;L^2(\O))}
		+ \| \fp - \fp_h  \|_{L^2(0,T;L^2(\O))}
		\lesssim \Rrate(h),\\
		\label{main1}\| u - u_h \|_{L^\infty(0,T;L^2(\O))} \lesssim \Rrate(h),\\
		\label{main2}\| \fu - \fu_h \|_{L^2(0,T;L^2(\O))} \lesssim \Srate(h),\\
		\label{main1-jump-phi}	h^{1/2}	 \| \p_h - \hp_h  \|_{L^2(0,T;L^2(\Eh))} \lesssim \Rrate(h),\\
		\label{main1-jump}h^{1/2-\kappa} \| \PiE{\kj}(u_h - \hu_h) \|_{L^2(0,T;L^2(\Eh))} \lesssim \Srate(h),
	\end{gather}
	with
	\begin{gather}
		\label{def:Rkk}\Rrate(h):=
		h^{\sbar-\epsB}\|u\|_{H^1(0,T;H^s(\O))}
		+h^{\sigbar+\kappa-\epsB}\|\fu\|_{H^1(0,T;H^\jsmooth(\O))}
		+h^{\rbar-1}\|\p\|_{H^1(0,T;H^{\psmooth}(\O))},\\
		\label{def:Skk}\Srate(h):=
		h^{\sbar-\kappa}\|u\|_{H^1(0,T;H^s(\O))}
		+h^{\sigbar}\|\fu\|_{H^1(0,T;H^\jsmooth(\O))}
		+h^{\rbar-1}\|\p\|_{H^1(0,T;H^{\psmooth}(\O))}.
	\end{gather}
	In the mixed Neumann--Robin case with $\kj=k-1$, the hidden constant may depend on
	$\varepsilon$; the hidden constants may depend on the data and on
	$\|\fp\|_{W^{1,\infty}(0,T;L^\infty(\O))}$ and
	$\|u\|_{L^\infty(0,T;L^\infty(\O))}$, but are independent of $h$.
\end{theorem}

\

\begin{table}[htpb]
{\centering
		\begin{tabular}{c|cc|cc} &
	\multicolumn{2}{c}{Pure Robin or Neumann b.c.} & \multicolumn{2}{c}{Mixed Neumann--Robin b.c.} \\
 & $\kj=k-1$ & $\kj = k $ & $\kj = k-1$ & $\kj = k$	\\
	\midrule
$\mathcal{R}_{\kj,k}$ & $h^{k+1}$ & $h^{k+1}$ & $h^{k+1-\varepsilon}$ & $h^{k+1}$	\\	
$\mathcal{S}_{\kj,k}$ & $h^k$ & $h^{k+1}$ & $h^{k}$ & $h^{k+1}$	\\
	\hline
\end{tabular}
\caption{
Convergence order attained by the two schemes under sufficiently high smoothness of the continuous solution. 
}
\label{table:HDG}
}
\end{table}

\section{Proof of Theorem \ref{thm:KSerror}}\label{sec:proof}

In this section, we present the proof of Theorem \ref{thm:KSerror}. 
The error estimate mostly follows, 
step by step, the proof of the analogous result in \cite{chen2019hdg}, with some differences due to the presence of the time derivative in equation \eqref{KS2} and the mixed Neumann--Robin boundary conditions. In addition, with respect to such a result (that only considers the choice $k'=k-1$), the different definition of the flux unknown $\fu$ allows us to obtain an optimal result also for $k'=k$. 
The proof is rather technical, and therefore we report it in full in order to clearly highlight how these differences affect the argument and ultimately modify the final result. In particular, the suboptimality, for $k'= k-1$, of the bounds \eqref{main1} and \eqref{main3} stems from the presence of mixed boundary conditions. We  remark that the estimates we derive are not robust with respect to the parameters of the equation (in particular $\mu$ and $\nu$), as the constants implicit in the bound of Theorem \ref{thm:KSerror} generally depend on all the parameters involved, as well as on the $L^\infty$ norm of $\nabla \p$. In the proof we keep at least partial track of the influence of these parameters, as we believe that this may provide the reader with useful information.

\smallskip

\subsection*{Scheme of the proof.} 
To facilitate following the proof, let us start by giving an idea of its overall structure. As in \cite{chen2019hdg}, the proof of Theorem \ref{thm:KSerror} consists of several steps.
\begin{enumerate}
\item The first step will consist in introducing different intermediate approximations to the true solution of the system \eqref{KS1}--\eqref{KS2} at a fixed time $t$, with the related error bounds. 
As far as $\Phi = (\p,\hp,\fp)$ is concerned, we will consider the projection, which depends on the stabilization parameter, specifically proposed in 
\cite{cockburn2010projection} 
for the analysis of HDG methods.   As far as $U = (u,\hu,\fu)$ is concerned, we will consider its $L^2$ projection $\PiLdue(u,\hu,\fu) \in \Vh$ (for $k'=k-1$), and, once again, the stabilization-parametric HDG projection (for $k'=k$). 
For $U$ we will also introduce a suitably defined ``elliptic projection'' $\UIh = (\uIh,\huIh, \fuIh) \in \Vh$, for which suitable approximation bounds will be proven. We remark that this is the only step where the proofs for $k'=k-1$ and $k'=k$ diverge.

\item The second step consists of analyzing the error of the $\Phi$ component by looking at the equations \eqref{KS-HDG2}, \eqref{KS-HDG4} and \eqref{KS-HDG5}, while considering $u$ and $u_h$ as given parameters. This yields an error bound for $\Phi$ in terms of the mesh size parameter and of $u-u_h$. 

\item Letting $\Tstarh$ denote the largest time $T' \leq T$ such that 	\begin{equation*}	\| u_h \|_{L^\infty(0,T';L^\infty(\O))} \leq  \Cinftyu\quad \text{ with } \quad\Cinftyu =	\| u \|_{L^\infty(0,T;L^\infty(\O))} + 2\end{equation*}
(for $h$ sufficiently small, $\Tstarh$ can be seen to be strictly positive), we prove that for all $t<\Tstarh$ it holds that	\[
\| u_h(t) - \uIh(t) \|_{\O}^2  \lesssim \Rrate(h)^2 + \int_0^{t } \| u_h(s) - \uIh(s) \|_{L^2(\O)}^2 \,ds 
\]
which, by Gronwall's Lemma, yields
\[	\| u_h - \uIh \|_{L^\infty(0,\Tstarh;L^2(\O))}^2 \lesssim \Rrate(h)^2.\]
\item To conclude, we will show that for $h$ sufficiently small, it holds that $\Tstarh = T$.
\end{enumerate}

\

We start by observing that, under our assumptions, Sobolev's embedding theorems imply that 
	\begin{equation}\label{eq:Linty-regularity}
	u\in L^\infty(0,T;L^\infty(\O)),
	\qquad
	\fp\in W^{1,\infty}(0,T;L^\infty(\O)^2).
	\end{equation}

It will be convenient, in the following to rewrite  \eqref{KS-HDG2}--\eqref{KS-HDG6} in compact form. To this aim we introduce bilinear forms  $\Au, \Ap: \V \times \V \to \mathbb{R}$, where $\V = \HTh{H^1(\Th)} \times L^2(\Eh) \times (\HTh{H^1(\Th)})^2$, defined, for $U = (u,\hu,\fu)$, $W = (w,\hw,\fw)$, $\Phi = (\p,\hp,\fp)$ and $\Theta = (\t,\ct,\ft)$ as
\begin{multline}\label{A.1}
\Au(U,W) = 
(\imu \fu ,\fw)_{\Omega}
+	  (\hfu \cdot \n, w)_{\Eh} - 
(\fu ,\nabla w)_{\Th} 	\\
- (u ,
\nabla \cdot \fw)_{\Th}+  (\hu , \fw\cdot \n)_{\Eh}-( \hfu  \cdot \n, \hw)_{\Eh} + \Ku(\hu,\hw)_{\GR},
\end{multline}
\begin{multline}
\Ap(\Phi,\Theta)  = (a \p,\t)_\Omega  +  (\inu \fp,\ft)_{\Omega} +
(\hfp\cdot\n, \t)_{\Eh}  -	(  \fp,\nabla\t)_{\Th}  	\\ - (\p ,
\nabla \cdot \ft)_{\Th}+ (\hp , \ft\cdot \n)_{\Eh} -  ( \hfp \cdot \n,\ct)_{\Eh} + \Kp (\hp,\ct)_{\GR},
\end{multline}
as well as the trilinear form  $\C(\cdot,\fp;\cdot): L^2(\O) \times L^\infty(\Th)^2 \times L^2(\O)^2 \to \mathbb{R}$ and the linear form $\F: \V \times \V \to \mathbb{R}$, respectively defined as
\begin{align}
\label{C}
&\C (u,\fp;\fw) = ( \inu \ctsens u \fp ,\fw)_{\Th}, \\
& \F(W,\Theta) = 	(f_u,w)_\Omega  + 	(f_\p,\t)_\Omega + (\gRu,\hw)_{\GR} + (\gRp,\ct)_{\GR} {+ (\gNu,\hw)_{\GN} + (\gNp,\ct)_{\GN}}.
\end{align}

\

\noindent Then, the HDG discretization of the Keller--Segel system rewrites as follows.
\begin{problem}\label{KS-HDG-compact} 
Find $U_h = (u_h, \hu_h, \fu_h) \in \Vh$, $\Phi_h =(\p_h,\hp_h,\fp_h) \in \Sh$
such that for all $W = (w,\hw,\fw) \in \Vh$, $\Theta = (\t,\ct,\ft) \in \Sh$ it holds that
\begin{multline}
(\dert u_h,w)_{\Omega} +
\dert (
\p_h,\t)_\Omega 
+ \Au(U_h,W) + \Ap(\Phi_h,\Theta) - (b u_h,\t)_\Omega + \C(u_h,\fp_h;\fw) \\= \F(W,\Theta).
\end{multline}
\end{problem}

\subsection{Preliminary bounds}

In order to analyze the method we introduce discrete norms and seminorms, defined as follows
\begin{gather}\label{A.5}
| U |_{\Au}^2 = \Au(U,U),\qquad 
\| \Phi \|_{\Ap}^2 =  
\Ap(\Phi,\Phi).
\end{gather}

By elementwise integration by parts, the following proposition is also not difficult to prove.
\begin{proposition}
It holds that
\begin{gather}
\label{A.11}	| U |_{\Au}^2 = \imu \| \fu \|_{\Omega}^2 + \tu h^{-\kappa} \|\PiE{k'} (u -\hu)  \|_{\Eh}^2 + \Ku \| \hu \|_{\GR}^2,\\
\| \Phi \|_{\Ap}^2 =  \inu \| \fp \|_{\Omega}^2 + \tp \| \p - \hp \|_{\Eh}^2  + \Kp \| \hp \|_{\GR}^2 + a \| \p \|^2_{\O}.
\end{gather}
\end{proposition}
We observe that $| \cdot |_{\Au}$ is a seminorm, while $\| \cdot \|_{\Ap}$ is a norm, since $a > 0$.
We also recall that, for $x \in H^1(0,T;X)$, $y \in X$, $X$ Hilbert space with scalar product $(\cdot,\cdot)_X$, it holds that (\cite[Section 5.9]{evans2022partial})
\begin{equation*}
\dt(x,y)_X = (\dt x, y)_X.
\end{equation*}

\subsection{Approximation bounds} We devote this section to introducing  some projection/quasi-interpolation operators that we will need in the following, and to  proving/recalling their properties. Throughout this and the next sections, we will write $\pm Q$ to indicate adding and subtracting the quantity $Q$.  We will consider the following projectors.

\subsubsection{The stabilization-parametric HDG projection onto
	$\Poly{k}(\Th)\times(\Poly{k}(\Th))^2$} \label{sec:HDGProj}
For $K\in\Th$ and  $\gamma$ positive, we define
the local HDG projector
\[
\PiHDG^\gamma(v,\boldsymbol q)
:=\bigl(\PiHs{\gamma}(v,\boldsymbol q),
\PiHv{\gamma}(v,\boldsymbol q)\bigr)
\in\Poly{k}(K)\times(\Poly{k}(K))^2
\]
as proposed in \cite{cockburn2010projection}.  Its two components are
determined by the elementwise coupled conditions
\begin{align}
	\bigl(\PiHv{\gamma}(v,\boldsymbol q),\boldsymbol r\bigr)_K
	&= (\boldsymbol q,\boldsymbol r)_K
	&&\forall\boldsymbol r\in(\Poly{k-1}(K))^2,
	\label{proj1}\\
	\bigl(\PiHs{\gamma}(v,\boldsymbol q),z\bigr)_K
	&= (v,z)_K
	&&\forall z\in\Poly{k-1}(K),
	\label{proj2}\\
	\left(
	\PiHv{\gamma}(v,\boldsymbol q)\cdot\n
	+\gamma\PiHs{\gamma}(v,\boldsymbol q),\widehat z
	\right)_{\bK}
	&=
	\left(
	\boldsymbol q\cdot\n+\gamma v,\widehat z
	\right)_{\bK}
	&&\forall\widehat z\in\Poly{k}(\EhK).
	\label{proj3}
\end{align}
For a triple $X=(v,\widehat v,\boldsymbol q)$, we use the same symbol for
the corresponding global projection and set
\begin{equation}\label{eq:defHDG-glob-proj}
	\PiHDG^\gamma(X)
	:=\bigl(\PiHs{\gamma}(v,\boldsymbol q),
	\PiE{k}(\widehat v),
	\PiHv{\gamma}(v,\boldsymbol q)\bigr),
\end{equation} where
$\PiE{k}$ is the $L^2$ projection onto $\Poly{k}(\Eh)$ introduced after \eqref{numflux}.
In particular, the projection used to approximate the chemoattractant variables
is
\[
\PiHDG^{\tp}(\Phi)
=\bigl(\PiHs{\tp}(\p,\fp),\PiE{k}(\hp),\PiHv{\tp}(\p,\fp)\bigr),
\qquad \Phi=(\p,\hp,\fp).
\]

Provided $\gamma \simeq 1$, we have the following interpolation bounds \cite[Theorem 2.1]{cockburn2010projection}: for
$\p \in \HTh{H^r(\Th)}$, $1 \leq r$, $\fp \in \HTh{H^{\rho}(\Th)}$, $1 \leq \rho $
\begin{gather}
	\| \p - \PiHs{\gamma}(\p,\fp) \|_{\O} \lesssim h^{\min \{r,k+1\}} \| \p \|_{\HTh{H^{r}(\Th)}}
	+  h^{\min \{\rho,k+1\}} \| \nabla \cdot \fp \|_{\HTh{H^{\rho-1}(\Th)}},\\
	\| \fp - \PiHv{\gamma}({\p},\fp) \|_{\O} \lesssim  h^{\min \{\rho,k+1\}}  \| \fp \|_{\HTh{H^{\rho}(\Th)}} +  h^{\min \{r,k+1\}}  \| \p \|_{\HTh{H^r(\Th)}}.
\end{gather}
Applying a standard trace bound we also have
\begin{multline}\label{boundEhPro}
	\| \p - \PiHs{\gamma}(\p,\fp) \|_{\Eh}^2  \lesssim h^{-1}\| \p - \PiHs{\gamma}(\p,\fp)\|^2_{\O} + h\|\nabla (\p - \PiHs{\gamma}(\p,\fp)) \|^2_{\Th} \\
	\lesssim h^{-1}\| \p - \PiHs{\gamma}(\p,\fp)\|^2_{\O}+  h\|\nabla (\p - \widetilde \p) \|_{\Th}^2 + h^{-1}\| \widetilde \p - \PiHs{\gamma}(\p,\fp) \|_{\O}^2 \\
	\lesssim h^{2\min\{r,k+1\}-1} \| \p \|^2_{\HTh{H^r(\Th)}} + h^{2\min\{\rho,k+1\}-1 }\| \nabla \cdot \fp \|^2_{\HTh{H^{\rho-1}(\Th)}},
\end{multline}
where $\widetilde \p \in \Poly{k}(\Th)$ is the $\HTh{H^{1}(\Th)}$ projection (i.e. the elementwise $H^1$ projection onto $\Poly{k}(K)$ for each $K \in \Th$) with zero mean error on each element, and where we applied an inverse estimate and standard  approximation bounds.

\begin{corollary} Assuming that $\gamma \simeq 1$, the following bounds hold for all $\p \in H^{r}(\O)$, $2 \leq r$,
	and ${\fp=-\nu\nabla\p}$:
	{	\begin{gather}\label{4.22}
			\| \p - \PiHs{\gamma}(\p,\fp) \|_{\O} +  	\| \fp - \PiHv{\gamma}(\p,\fp) \|_{\O} \lesssim h^{\min\{r-1,k+1\}} \| \p \|_{H^r(\O)},\\
\label{4.22b}			\| \p - \PiHs{\gamma}(\p,\fp) \|_{\Eh} \lesssim h^{\min\{r-1,k+1\}-1/2} \| \p \|_{H^r(\O)}.
	\end{gather}}
\end{corollary}

We also recall that, by construction, the HDG projection satisfies the following Proposition, the proof of which is recalled  in Appendix~\ref{section:prop_projections}, for the sake of completeness.

\begin{proposition}\label{prop:HDGproj}
	For all $\Phi = (\p,\hp,\fp)  \in \HTh{H^1(\Th)} \times L^2(\Eh) \times \HTh{H^1(\Th)}^2$ and all ${\Theta = (\t,\widehat\t,\ft) \in \Sh}$ it holds that
	\begin{gather}	
		\Ap(\Phi - \PiHDG^{\tp}(\Phi),\Theta)	=  (a (\p-\PiHs{\gamma_\phi}(\p,\fp)),\t)_\Omega +  (\inu (\fp-\PiHv{\gamma_\phi}(\p,\fp)),\ft)_{\Omega}\label{ortHDG2}.
	\end{gather}
\end{proposition}

\subsubsection{The $L^2$ projection onto $\Poly{k}(\Th) \times \Poly{k-1}(\Eh) \times (\Poly{k-1}(\Th))^2$}
\label{sec:L2Proj}
  In the case $k'=k-1$, it will be convenient to also introduce the plain $L^2$ projector onto the space $\Vh$. 
We let $\PiO : L^2(\O) \to \Poly{k}(\Th)$, $ \fPiO: (L^2(\O))^2 \to (\Poly{k-1}(\Th))^2$ denote the standard $L^2$ projections, and we let $\PiLdue$ denote the component-wise $L^2$ projection:
\[
\PiLdue(u,\hu,\fu) = (\PiO u, \PiE{k-1} \hu, \fPiO\fu) \in \Poly{k}(\Th)  \times \Poly{k-1}(\Eh) \times ( \Poly{k-1}(\Th))^2.
\]
Standard approximation estimates  immediately yield the following bounds: $U = (u,\hu,\fu) \in \HTh{H^{s}(\Th)}\times   H^{\hs}(\Eh) \times (\HTh{H^\jsmooth (\Th)})^2$ with $0 \leq s $, $0 \leq \hs$, $1 \leq \jsmooth $ imply
\begin{gather}\label{standardL2}
	\| u - \PiO(u) \|_\O \lesssim h^{\min\{s,k+1\}} | u |_{\HTh{H^s(\Th)}},
	\quad 
	\| \hu - \PiE{k-1}(\hu) \|_{\Eh} \lesssim h^{\min\{\hs,k\}} | \hu |_{H^{\hs}(\Eh)},\\
	 \| \fu - \fPiO \fu \|_{\O} \lesssim h^{\min\{\jsmooth,k \}} | \fu |_{\HTh{H^\jsmooth(\Th)}}.
\end{gather}
For $\fu = -\mu\nabla u - \chi u \fp/\nu$, with $u \in  H^s(\O)$ and $\fp \in W^{s-1,\infty}(\O)^2$, this becomes
\begin{equation*}
 \| \fu - \fPiO \fu \|_{\O} \lesssim h^{\min\{s-1,k \}} \| u \|_{H^s(\O)},
\end{equation*}
the implicit constant depending on $\fp$. Moreover, $u \in \HTh{H^s(\Th)}$, 
$1 < s$ implies \cite{ern2021finite,dupont1980polynomial}
\begin{equation}\label{boundpiLinfty}
	\| u - \PiO u \|_{L^\infty(\O)} \lesssim h^{\min\{s-1,k\}} \| u \|_{\HTh{H^s(\Th)}}.
\end{equation}
Moreover, proceeding as in \eqref{boundEhPro}, we have
\begin{equation}
	\label{4.22L2} \| u - \PiO u \|_{\Eh} \lesssim h^{\min\{s,k+1\}-1/2} | u |_{\HTh{H^s(\Th)}}.
\end{equation}
If $u \in H^s(\O)$ with $1/2 < s < 1$, then a similar bound to \eqref{4.22L2} still holds. Indeed, using \eqref{trace_s} instead of \eqref{trace} in the argument leading to \eqref{4.22L2}, we have
\begin{equation}
	\| u - \PiO u \|^2_{\Eh} \lesssim h^{-1} \| u - \PiO u \|^2_{\O} + \|  u - \PiO u \|^{2-1/s}_{\O} | u - \PiO u |^{1/s}_{H^s(\Th)} \lesssim h^{2s-1} \| u \|^2_{H^s(\O)}.
\end{equation}
Analogously, for $\fw \in (H^s(\O))^2$, $1/2 < s < 1$, we have that
\[
\| \fw - \fPiO \fw \|_{\Eh}^2 \lesssim h^{2s-1} \| \fw \|_{(H^s(\O))^2}.
\]

The following proposition, whose proof we postpone to the Appendix, will be useful in the following for handling the case $k'=k-1$ ($\kappa = 1$).
\begin{proposition} \label{prop:L2} Let $k'=k-1$.
	For all $W = (w, \hw, \fw)\in \HTh{H^1(\Th)} \times L^2(\Eh) \times (\HTh{H^1(\Th)})^2$,  and all $U = (u,\hu,\fu)\in \Vh$, setting  $\PiLdue(W) = (w^\pi, \hw^\pi, \fw^\pi)$, it holds that 
	\begin{multline}
		\label{A.26}	\Au(U,W-\PiLdue(W)) = 	(\PiE{k-1}(u-\hu),\tu h^{-1}\PiE{k-1}(w-w^\pi))_{\Eh} \\ -(u - \hu,(\fw - \fw^\pi)\cdot \n )_{\Eh},
			\end{multline}
		\begin{multline}
		\Au(W - \PiLdue(W),U) = 	(\tu h^{-1}\PiE{k-1}(w-w^\pi),\PiE{k-1}(u - \hu))_{\Eh} \\ + ((\fw - \fw^\pi)\cdot \n,u - \hu)_{\Eh}.
	\end{multline}
\end{proposition}

\subsection{The elliptic projection onto $\Vh$}\label{sec:cellular-projection}
For $U = (u,\hu,\fu) \in H^1(\O) \times L^2(\Eh) \times (H^1{(\O)})^2$ we let $\UIh =(\uI,  \huI,\fuI) \in \Vh$ be such that for all $W = (w,\hw,\fw) \in \Vh$
\begin{equation}\label{pb:elliptic}
	M\, (\uI,w)_\Omega + \Au(\UIh,W) + \Cpsi(\uIh,\fw) = 	M\,(u,w)_\Omega + \Au(U,W) + \Cpsi(u,\fw),
\end{equation}
where $M > 0$ is a constant which we assume to be sufficiently large, and where $\Cpsi:L^2(\O) \times (L^2(\O))^2 \to \mathbb{R}$ is the bilinear form obtained by fixing the value of $\fp$ in the trilinear operator $\C$, that is
\(\Cpsi(u,\fw) = \C(u,\fp;\fw)\).
Problem \eqref{pb:elliptic} is well posed (see Section ~\ref{sec:proof_cell_proj}) and we have the
following lemma, whose proof, which  
is the only point where the two
schemes require different arguments, is postponed to Section  \ref{sec:proof_cell_proj}.

\begin{lemma}[Cellular projection estimates]\label{lem:cellular-projection}
	Let $\UIh$ solve \eqref{pb:elliptic}.  There exist $M_0 > 0$  independent of $h$ and $h_0>0$	such that, if $M > M_0$ and $h<h_0$, then, under the assumptions of Theorem \ref{thm:KSerror}, we have, uniformly in time,
	\begin{gather}
		\label{cellproj-u}\| u - \uI \|_{\O} \lesssim h^{\sbar-\epsB}\|u\|_{H^s(\O)} + h^{\sigbar+\kappa-\epsB}\|\fu\|_{H^\jsmooth(\O)},\\
		\label{cellproj-j}\| \fu - \fuI \|_{\O}  \lesssim h^{\sbar-\kappa}\|u\|_{H^s(\O)} + h^{\sigbar}\|\fu\|_{H^\jsmooth(\O)},\\
		\label{cellproj-hu} \| \PiE{\kj}(\uI - \huI)\|_{\Eh} \lesssim h^{\sbar-1/2}\|u\|_{H^s(\O)} + h^{\sigbar + \kappa - 1/2}\|\fu\|_{H^\jsmooth(\O)},\\
		\label{cellproj-dt}\|\dt u-\dt\uIh\|_\O
		\lesssim h^{\sbar-\epsB}(\|u\|_{H^s(\O)}+ \| \dt u \|_{H^s(\O)}) + h^{\sigbar+\kappa-\epsB}(\|\fu\|_{H^\sigma(\O)} + \|\dt \fu\|_{H^\sigma(\O)}).
	\end{gather}
	For $s>k+1$ or $\jsmooth>\kj+1$ the approximation results above
	are applied with the capped indices $\sbar$, $\sigbar$ in place of
	$s$, $\jsmooth$.
\end{lemma}

\begin{remark}
	For $\kj= k$ the proof of Lemma~\ref{lem:cellular-projection} relies on an energy argument only: neither the convexity of $\Omega$ nor the smallness of $h$ is needed, and the bounds 	\eqref{cellproj-u}-\eqref{cellproj-dt} hold on any polygonal domain. Convexity is only used, for $\kj = k -1$, in the duality argument leading to  \eqref{cellproj-u}. 
\end{remark}

\subsection{Estimating $\Phi-\Phi_h$ as a function of $u-u_h$.} Assuming that $\phi \in H^1(0,T;H^\psmooth(\O))$, $2 < \psmooth$, we are now able to prove the following lemma (where, we recall, $\hp = \p|_{\Eh}$ and $\fp = - \nu \nabla \p$).
\begin{lemma}\label{lemma4.1} For all $T'$ with $[0,T']\subset J_h$ it holds that
\begin{multline}
\| \p - \p_h \|_{L^\infty(0,T';L^2(\O))}^2 +   \| \p - \p_h \|_{L^2(0,T';L^2(\O))}^2 
+  \| \fp - \fp_h \|_{L^2(0,T';L^2(\O))}^2  \\ 
+ 	\tp \| (\p_h - \hp_h) - (\PiHs{\tp}(\p,\fp ) - \PiE{k}\hp) \|_{L^2(0,T';L^2(\Eh))}^2 + K_\p \| \PiE{k} \phi - \hp_h \|_{L^2(0,T';L^2(\GR))}^2\\
\lesssim	 
h^{2(\rbar -1)} \| \p \|^2_{H^1(0,T';H^{\psmooth}(\O))}
+	\| u- u_h \|_{L^2(0,T';L^2(\O))}^2.
\end{multline}
\end{lemma}

\begin{proof}
Let $\Delta = (\myeta,\heta,\feta) = \Phi_h - \PiHDG^{\tp}(\Phi)$. We start by observing that we have
\[
\|\myeta(0)\|_\O=\|\PiO\p_0-\PiHs{\tp}(\p_0,\fp(0))\|_\O\lesssim h^{\rbar-1}\|\p\|_{H^1(0,T;H^\psmooth(\O))}.
\]
It holds that
\begin{multline*}
(\dt \myeta,\myeta)_\O + \inu \| \feta \|_{\Omega}^2 + \tp \| \myeta - \heta \|_{\Eh}^2 + a \| \myeta \|_{\O}^2 + \Kp \| \heta \|_{\GR}^2 = (\partial_t  \myeta,\myeta)_\O + \Ap(\Delta,\Delta)
\\ = \underbrace{(\partial_t  \p_h - \dt \p,\myeta)_\O+\Ap(\Phi_h - \Phi,\Delta) }_{=b(u_h-u,\myeta )_\O}+(\partial_t  \p - \dt \PiHs{\tp}(\p,\fp),\myeta) +\underbrace{ \Ap(\Phi - \PiHDG^{\tp}(\Phi),\Delta)}_{\text{use \eqref{ortHDG2}}} = \\
b(u_h - u, \myeta)_\O +(\partial_t  \p - \dt \PiHs{\tp}(\p,\fp),\myeta) + (a (\p-\PiHs{\tp}(\p,\fp)),\myeta)_\Omega +  (\inu (\fp-\PiHv{\tp}(\p,\fp)),\feta)_{\Omega}\\
\leq
\frac{ 3 | b |^2} {2 a}\| u - u_h \|_{\O}^2  +
\frac  3 {2a} \| \partial_t  \p - \dt \PiHs{\tp}(\p,\fp) \|_{\O}^2 + 
\frac {3 a} 2 \| \p-\PiHs{\tp}(\p,\fp) \|_{\O}^2 + \frac \inu 2 \| \fp - \PiHv{\tp}(\p,\fp) \|_{\O}^2 \\
+
\frac { a} 2 \| \myeta \|^2 + \frac \inu 2 \| \feta \|_{\O}^2 
\end{multline*}
finally yielding
\begin{multline}
(\dt \myeta,\myeta)_\O  + \frac \inu 2 \| \feta \|_{\Omega}^2 + \tp \| \myeta - \heta \|_{\Eh}^2 + \frac a 2 \| \myeta \|_{\O}^2 + \Kp  \| \heta \|_{\GR}^2 \\
\leq
\frac{ 3 | b |^2} {2 a}\| u - u_h \|_{\O}^2 +
\frac  3 {2a} \| \partial_t  \p - \dt \PiHs{\tp}(\p,\fp) \|_{\O}^2  + \frac {3a} 2 \| \p-\PiHs{\tp}(\p,\fp) \|_{\O}^2 + \frac \inu 2 \| \fp - \PiHv{\tp}(\p,\fp) \|_{\O}^2.
\end{multline}
Integrating  in $(0,t)$ for any $t \leq T$ we obtain
\begin{multline*}
\frac 1 2\| \myeta(t) \|_{\O}^2 + \frac \inu 2  \int_0^t  \| \feta(s) \|_{\O}^2\,ds + \tp \int_0^t   \| \myeta(s) - \heta(s) \|_{\Eh}^2\,ds + \frac a 2 \int_0^t   \| \myeta(s) \|_{\O}^2\,ds + \int_0^t   \Kp \| \heta(s) \|_{\GR}^2\,ds\\
\leq	\frac 1 2\| \myeta(0) \|_{\O}^2 +
\frac{ 3 | b |^2} {2 a}
\int_{0}^t
\| u(s) - u_h(s) \|_{\O}^2 \,ds 
+ \frac 3{2a} \int_0^t   \|\dt \p(s) - \dt \PiHs{\tp}(\p(s),\fp(s)) \|_{\O}^2\,ds
\\+ \frac {3 a} 2 \int_0^t   \| \p(s)-\PiHs{\tp}(\p(s),\fp(s)) \|_{\O}^2 \,ds + \frac \inu 2 \int_0^t   \| \fp(s) - \PiHv{\tp}(\p(s),\fp(s)) \|_{\O}^2\,ds.
\end{multline*}
By the arbitrariness of $t \leq T$ we then conclude that for any $T'\leq T$
\begin{multline*}
\frac 1 2\| \myeta \|_{L^\infty(0,T';L^2(\O))}^2 
+ \frac \inu 2 \| \feta \|_{L^2(0,T';L^2(\O))}^2 + \tp \| \myeta - \heta \|_{L^2(0,T';L^2(\Eh))}^2  + \frac a 2 \| \myeta \|_{L^2(0,T';L^2(\O))}^2 + \Kp \| \heta \|_{L^2(0,T';L^2(\GR))}^2\\
\lesssim	\frac 1 2\| \myeta(0) \|_{\O}^2 +
\frac{ 3 | b |^2} {2 a}
\| u- u_h \|_{L^2(0,T';L^2(\O))}^2 + 
\frac 3{2a} \|\dt (\p -  \PiHs{\tp}(\p,\fp)) \|_{L^2(0,T';L^2(\O))}^2
\\+ \frac {3a} 2  \| \p-\PiHs{\tp}(\p,\fp) \|_{L^2(0,T';L^2(\O))}^2  + \frac \inu 2 \| \fp - \PiHv{\tp}(\p,\fp) \|_{L^2(0,T';L^2(\O))}^2.
\end{multline*}
We can now combine this estimate with the estimates on the HDG projection and we obtain the desired result by a triangle inequality.
\end{proof}

\subsection{Estimating  $U_h - \UIh$}
Let $\Tstarh \in J_h$ denote the largest time, in the discrete solution existence interval given by Proposition \ref{prop:local-existence},  for which	\begin{equation}\label{condTstar}	\| u_h \|_{L^\infty(0,\Tstarh;L^\infty(\O))} \leq  	\| u \|_{L^\infty(0,T;L^\infty(\O))} + 2=: \Cinftyu.\end{equation}
It is easy to see that, for $h$ sufficiently small we have that $\Tstarh > 0$. Indeed, by \eqref{boundpiLinfty}, there exists an $h_1$ such that for $h \leq h_1$  we have that
\[
\| u(0) - \PiO u(0) \|_{L^\infty(\Omega)} \leq 1,
\]
whence
\begin{gather*}
\| u_h (0) \|_{L^\infty(\Omega)} = \| \PiO u(0) \|_{L^\infty(\Omega)}\leq \| \PiO u(0) - u(0) \|_{L^\infty(\Omega)} + \| u(0) \|_{L^\infty(\Omega)} \leq 1 +  \| u \|_{L^\infty(0,T;L^\infty(\O))} .
\end{gather*}
By the continuity in time of the solution of equation \eqref{KS-HDG1}, we have that $\Tstarh > 0$ for all $h \leq h_1$.
We have the following bound.
\begin{lemma}\label{lem:keybound} For all $t \leq \Tstarh$ it holds that
\begin{multline}\label{bound3}
\| u_h(t) - \uIh(t) \|_{\O}^2  \lesssim 	h^{2(\sbar-\epsB)}	{\| u \|^2_{H^1(0,T;H^s(\O))} } + h^{2(\sigbar +\kappa-\epsB)} \| \fu \|^2_{H^1(0,T;H^\jsmooth (\O))}  \\ + h^{2({\rbar}-1)} \| \p \|_{H^1(0,T;H^{\psmooth}(\O))}^2
+  \int_0^{t } \| u_h(y) - \uI(y) \|^2_{\O}\,dy, 	
\end{multline}
 the implicit constant in the inequality depending on $\| \fp \|_{L^\infty(0,T;L^\infty(\O))}$ and $\| u \|_{L^\infty(0,T;L^\infty(
	\O))}$.
\end{lemma}
By applying Gronwall Lemma, thanks to the arbitrariness of $t$, this yields, in turn, the following corollary.
\begin{corollary}  It holds that
\begin{multline}\label{bound4}
\| u_h - \uIh \|^2_{L^\infty(0,\Tstarh;L^2(\O))} \lesssim h^{2(\sbar-\epsB)}	{\| u \|^2_{H^1(0,T;H^s(\O))} } +  h^{2(\sigbar +\kappa-\epsB)} \| \fu \|^2_{H^1(0,T;H^\jsmooth (\O))}\\ + h^{2({\rbar}-1)} \| \p \|_{H^1(0,T;H^{\psmooth}(\O))}^2.
\end{multline}
\end{corollary}

\begin{proof}[Proof of Lemma \ref{lem:keybound}]
	Let
\[
\ConstPsi:=
|\inu\ctsens|\,\|\fp\|_{L^\infty(0,T;L^\infty(\O))}.\] By the definition of $\C$, we have, for $v \in L^2(\O)$, $\fv \in (L^2(\O))^2$
\begin{equation}\label{eq:continuityC}
|\C(v,\fp;\fv)|
\le
\ConstPsi\|v\|_{\O}\|\fv\|_{\O}.
\end{equation}
Now, for $E = (e,\he,\fe)$ with $e = u_h - \uI$, $\he = \hu_h - \huI$, $\fe = \fu_h - \fuI$, we have
\begin{multline*}
\frac 12 \partial_t  \| e \|_{\O}^2 + \alpha \| \fe \|_{\O}^2 + 
\tu h^{-\kappa}
{\|\PiE{\kj}(e - \he) \|_{\Eh}^2} + \Ku \| \he \|_{\GR}^2 =	(\dt e,e) + \Au(E,E)\\ =
(\dt u_h - \dt u,e)   + (\dt u - \dt \uI,e) + \Au( U_h-U,E ) + \Au(U - \UIh,E)\\ \pm \C(u_h,\fp_h;\fe) \pm \C(u_h,\fp;\fe) \pm \C(u,\fp;\fe) \pm \C(\uI,\fp;\fe) \pm M\,(u,e) \pm M\,(\uI,e) 
\\= 	\underbrace{(\dt u_h,e) 
+	 \Au( U_h,E)
+ \C(u_h,\fp_h;\fe) }_{=(f_u,e) + (\gRu,\he)_{\GR}{+ (\gNu,\he)_{\GN}}}
\underbrace{- (\dt u,e)   
-\Au(U,E ) 
- \C(u,\fp;\fe) }_{= - (f_u,e)_\O - (\gRu,\he)_{\GR} {- (\gNu,\he)_{\GN}}}
\\		\underbrace{+ M\,(u - \uI,e) 	 + \Au(U - \UIh,E )	+ \C(u - \uI,\fp;\fe)}_{=0}\\
+	(\dt (u - \uI),e) + M\,(\uI - u,e)  - \C(e,\fp;\fe)
+ \C(u_h, \fp - \fp_h;\fe) 
\\
= (\dt (u - \uI),e) + M\,(\uI - u,e)  - \C(e,\fp;\fe)
+ \C(u_h, \fp - \fp_h;\fe) 
\\
\leq \| \dt u - \dt \uI \|_{\O} \| e \|_{\O}   +
M \| \uI - u \|_{\O} \| e  \|_\O + \ConstPsi \| e \|_\O \| \fe \|_\O
+ | \C(u_h,\fp - \fp_h ;\fe)|  \\
\leq 
\| \dt u - \dt \uI \|^2_{\O} +  {M^2} \| \uI - u \|_{\O}^2 + \frac 1 2 \| e \|^2_{\O} +  \frac \alpha 4 \| \fe \|^2 + \frac {\ConstPsi^2} \alpha \| e \|^2_\O +
|\C(u_h,\fp_h - \fp;\fe)|.
\end{multline*}
Subtracting $\alpha \| \fe \|_\Omega^2/4$ from both sides and integrating in $(0,t)$ we obtain
\begin{multline}
\frac 12 \| e(t) \|_{\O}^2 + \frac 3 4 \alpha \| \fe \|_{L^2(0,t;L^2(\O))}^2 \\ + \tu  h^{-\kappa} \|\PiE{\kj}(e - \he) \|^2_{L^2(0,t;L^2(\Eh))} + \| \he \|_{L^2(0,t;L^2(\GR))}^2\\ \lesssim \frac 12 \| e(0) \|_{\O}^2 +
\| \dt u - \dt \uI \|^2_{L^2(0,t;L^2(\O))} 
+  {M^2}  \| \uI - u \|_{L^2(0,t;L^2(\O))}^2
\\ + \int_0^{t} \| e \|^2_{\O}  +
\int_0^{t} | \C(u_h,\fp - \fp_h;\fe)|,
\end{multline}
where the implicit constant depends on $\| \fp \|_{L^\infty(0,T;L^\infty(\O))}$.

\

We now claim that for all $t \leq \Tstarh$ it holds that
\begin{multline}\label{boundC}
\int_0^{t } | \C(u_h,\fp - \fp_h;\fe)  | \\
\leq \epsilon \| \fe \|^2_{L^2(0,t;L^2(\O))}   
+\frac {C(u)} \epsilon \left( h^{2(
\rbar-1)} \| \p \|^2_{H^1(0,t;H^\psmooth(\O))} + \| u - u_h  \|^2_{L^2(0,t;L^2(\O))} \right)
\end{multline}
with $C(u)\simeq\beta^2\ctsens^2(\Cinftyu)^2$ depending on  $	\| u \|^2_{L^\infty(0,T;L^\infty(\O))}$ and $\epsilon > 0$ arbitrary. Provided \eqref{boundC} holds, choosing $\epsilon = \alpha/4$ and subtracting $\alpha/4 \| \fe \|_{L^2(0,t;L^2(\O))}^2$ from both sides, and   using \eqref{standardL2}, Lemma~\ref{lem:cellular-projection} as well as $\tu \simeq 1$ we then obtain
\begin{multline}\label{A.48}
\frac 12 \| e(t) \|_{\O}^2 + \frac 1 2 \alpha \| \fe \|_{L^2(0,t;L^2(\O))}^2 + h^{-\kappa} \|\PiE{\kj}(e - \he) \|^2_{L^2(0,t;L^2(\Eh))} + \| \he \|_{L^2(0,t;L^2(\GR))}^2\\ 
\lesssim h^{2(\sbar-\epsB)} \| u \|^2_{H^1(0,T;H^s(\O))} + h^{2(\sigbar +\kappa-\epsB)} \| \fu \|^2_{H^1(0,T;H^\jsmooth (\O))}
\\	+  \int_0^{t } \| e \|^2_{\O} 
+\frac {4 C(u)} \alpha h^{2(\rbar -1)} \| \p \|_{H^1(0,t;H^\psmooth(\O))}^2,
\end{multline}
which implies the desired result, where we used that $ u - u_h = u - \uIh - e$ and Lemma \ref{lem:cellular-projection}.

It remains to prove \eqref{boundC}. For $t' \in [0,t]$, with $t \leq \Tstarh$, so that \eqref{condTstar} holds for $T' = t$
we can write
\begin{gather*}
|	\C(u_h(t'),\fp(t') - \fp_h(t');\fe(t'))| 
\leq 
	\inu |\ctsens |	\| u_h(t') \|_{L^\infty(\O)} 	\| \fp(t') - \fp_h(t') \|_{\O}  \| \fe(t') \|_{\O}.
\end{gather*}
We integrate between $0$ and $t$ and, using \eqref{condTstar}, we obtain, for $\epsilon > 0$ arbitrary
\begin{multline}\label{bound1}
\int_0^{t} |	\C(u_h(t'),\fp(t') - \fp_h(t');\fe(t'))| \,dt'\\
\leq \inu | \ctsens | 
\Cinftyu
\int_0^{t } 
\| \fp(t') - \fp_h(t'	) \|_{\O}  \| \fe(t') \|_{\O} \,dt'\\
\leq \inu | \ctsens | 
\Cinftyu\| \fp- \fp_h \|_{L^2(0,t;L^2(\O))} \| \fe \|_{L^2(0,t;L^2(\O))}\\ \leq
\epsilon \| \fe \|_{L^2(0,t;L^2(\O))}^2 +  \frac {\inu^2} {4 \epsilon}   | \ctsens |^2 
(\Cinftyu)^2 \| \fp- \fp_h \|^2_{L^2(0,t;L^2(\O))}.
\end{multline}
Using Lemma \ref{lemma4.1} to bound the last term in \eqref{bound1} yields \eqref{boundC}.
\end{proof}

\subsection{The continuation argument} We finally prove that, for $h$ sufficiently small, we have that $\Tstarh = T$. This is the object of the following lemma.
\begin{lemma}\label{lem:continuation}
There exists $h^*>0$ such that for all $h \le h^*$ it holds that $\Tstarh = T$.
\end{lemma}

\begin{proof}We once again apply the reasoning of  \cite{chen2019hdg}, which proves the result by contradiction.
If the statement of the lemma is false, there exists a sequence $h_n \to 0$ such that $\Tstarhn < T$.	By definition of $\Tstarhn$ we have that
\begin{equation*}
\| u_{h_n}(t) \|_{L^\infty(\Omega)} 
\leq \| u \|_{L^\infty(0,T;L^\infty(\O))} + 2, \quad t \le \Tstarhn,
\end{equation*}
and, for values of $t > \Tstarhn$ arbitrarily close to $\Tstarhn$
\begin{gather*}
\| u_{h_n}(t) \|_{L^\infty(\Omega)} > \| u \|_{L^\infty(0,T;L^\infty(\O))} + 2.
\end{gather*}
Then, by the continuity in $t$ of the solution to equation \eqref{KS-HDG1} we have
\[
\| u_{h_n}(\Tstarhn) \|_{L^\infty(\Omega)} = \| u \|_{L^\infty(0,T;L^\infty(\O))} + 2.
\]
At $t=\Tstarhn$, the inverse inequality in two
space dimensions, Lemma~\ref{lem:cellular-projection}, and
\eqref{bound4} give
\begin{equation*}
	\|u_{h_n}(t)\|_{L^\infty(\O)}
	\le \|u(t)\|_{L^\infty(\O)}
	+C h_n^{-1}\Rrate(h_n)
	+\|u(t)-u_{h_n}^I(t)\|_{L^\infty(\O)}.
\end{equation*}
Since $s>2$, $\jsmooth+\kappa>1$, and $\psmooth>2$, we have
\[
\min\{\sbar-\epsB,\,\sigbar+\kappa-\epsB,\,\rbar-1\}>1,
\]
and hence
\[
h_n^{-1}\Rrate(h_n)
\lesssim
h_n^{\sbar-1-\epsB}
+h_n^{\sigbar+\kappa-1-\epsB}
+h_n^{\rbar-2}
\longrightarrow0.
\]
Moreover, inserting the elementwise $L^2$ projection of $u$ in the last
term and using the approximation and inverse inequalities gives
\[
\|u(t)-u_{h_n}^I(t)\|_{L^\infty(\O)}
\lesssim
h_n^{\sbar-1}
+h_n^{-1}\Rrate(h_n)
\longrightarrow0.
\]
This contradicts the defining equality
\[
\|u_{h_n}(\Tstarhn)\|_{L^\infty(\O)}
=
\|u\|_{L^\infty(0,T;L^\infty(\O))}+2.
\]
Hence $\Tstarh=T$ for all sufficiently small $h$.
\end{proof}

Lemma \ref{lem:continuation} implies that, for $h\leq h^*$ the maximal discrete solution existence interval $J_h$ satisfies $J_h = [0,T]$. Indeed, if $J_h = [0,T_h)$ with $T_h \leq T$, necessarily $\| u_h(t) \|_{L^\infty(\O)}$ must blow up for $t$ going to $T_h$, and, consequently $\Tstarh < T_h$, which contradicts $\Tstarh = T$.

\

For $h \le  h^*$, we can  set $\Tstar = \Tstarh = T$ in bounds \eqref{bound3} and \eqref{bound4}, and  we obtain the following corollary.
\begin{corollary}  For $h \leq h^*$ it holds that
\begin{multline}\label{bound4bis}
\| u_h - \uIh \|^2_{L^\infty(0,T;L^2(\O))} \lesssim h^{2(\sbar-\epsB)}	{\| u \|^2_{H^1(0,T;H^s(\O))} } \\+  h^{2(\sigbar +\kappa-\epsB)} \| \fu \|^2_{H^1(0,T;H^\jsmooth (\O))} + h^{2({\rbar}-1)} \| \p \|_{H^1(0,T;H^{\psmooth}(\O))}^2.
\end{multline}
\end{corollary}
Combining this estimate with Lemma \ref{lemma4.1} and with the HDG projection bounds, by triangle inequality and using \eqref{4.22b} and the approximation properties of $\PiE{k}$, we  obtain the following bound.
\begin{corollary}\label{cor:phi-bound}
For $h \le h^*$ it holds that
\begin{gather}
\| \p - \p_h \|_{L^\infty(0,T;L^2(\O))} + \| \fp-\fp_h \|_{L^2(0,T;L^2(\O))} + h^{1/2} \| \p_h - \hp_h \|_{L^2(0,T;L^2(\Eh))}
{\lesssim  \Rrate(h).}
\end{gather}
\end{corollary}
Additionally, bound \eqref{A.48} yields
\begin{corollary} For $h \leq h^*$ it holds that
	{
\begin{gather}
\| \fu_h-\fuI \|_{L^2(0,T;L^2(\O))}^2
+ h^{-\kappa} \|\PiE{\kj}(u_h - u_h^I ) - \hu_h + \hu_h^I \|^2_{L^2(0,T;L^2({\Eh}))} 
\lesssim   \Rrate(h)^2.
\end{gather} }
\end{corollary}
{
The chemoattractant estimate and \eqref{main1} follow by combining the
preceding corollaries with the projection bounds.  For the cellular flux, we
use
\[
\fu-\fu_h=(\fu-\fuI)-(\fu_h-\fuI).
\]
Lemma~\ref{lem:cellular-projection} gives
\[
\|\fu-\fuI\|_{L^2(0,T;L^2(\O))}
\lesssim
h^{\sbar-\kappa}\|u\|_{H^1(0,T;H^s(\O))}
+h^{\sigbar}\|\fu\|_{H^1(0,T;H^\jsmooth(\O))},
\]
whereas the last corollary controls $\fu_h-\fuI$.  Since $0<h<h^*$ and
$0\le\epsB<1/2$, the higher powers
$h^{\sbar-\epsB}$ and $h^{\sigbar+\kappa-\epsB}$ are absorbed into the
corresponding $h^{\sbar-\kappa}$ and $h^{\sigbar}$ terms with a constant independent
of $h$ (for $\kj=k$ the two rates coincide and no absorption is needed).  This proves \eqref{main2}.  Finally,
\[
\PiE{\kj}(u_h-\hu_h)=\PiE{\kj}(u_h-\uIh)+\PiE{\kj}(\uIh-\huIh)+\PiE{\kj}(\huIh-\hu_h).
\]
The last corollary controls the sum of the first and third terms through
$h^{-\kappa/2}\|\PiE{\kj}(e-\he)\|_{\Eh}$ (up to the uniformly bounded factor $h^{(1-\kappa)/2}\le1$), whereas \eqref{cellproj-hu} controls the middle
term with the energy-order rate.  The triangle inequality therefore gives
\eqref{main1-jump}, while \eqref{main1-jump-phi} follows from Corollary~\ref{cor:phi-bound}.  This completes the proof of
Theorem~\ref{thm:KSerror}.}


\subsection{Proof of Lemma \ref{lem:cellular-projection} }\label{sec:proof_cell_proj}
Letting $\ConstPsi = | \inu \widetilde \chi | \| \fp \|_{L^\infty(0,T;L^\infty(\O))}$, the coercivity on $\Vh$ of the bilinear form 
\[
M(u,w)_\Omega + \Au(U,W) + \Cpsi(u,\fw)
\]
for $M > 4 \ConstPsi^2/\imu$  descends directly from \eqref{A.5} and \eqref{A.11}. Existence and uniqueness of the solution immediately follow.
 To bound \eqref{cellproj-u} and \eqref{cellproj-j},
we observe 
\begin{equation}\label{elliptic_orthogonality}
	M(u-u^I_h,w) + \Au(U-U^I_h,W) + \Cpsi(u-u^I_h,\fw) = 0, \quad \text{for all }W = (w,\hw,\fw) \in \Vh.
\end{equation}
Letting $U^\pi = (u^\pi,\hu^\pi,\fu^\pi) \in \Vh$ denote an approximation to the true solution, and letting $E = (e, \he,\fe) = U^\pi - U_h^I$ we can write
\begin{multline*}
	M \| e \|_{\O}^2 +  | E |_{\Au}^2 = M(e,e)_{\O} + \Au(E,E) =
	M (u^\pi - u, e)_\O + \Au(U^\pi-U,E)
	\\ 	 + M ( u - \uIh,e) 
	+ \Au(U - \UIh,E)	\\ 	=	 M (u^\pi-u,e)_\O + \Au(U^\pi-U,E) - \Cpsi(u-\uIh,\fe) = I + II + III.
\end{multline*}
We  then need to bound $I$, $II$ and $III$. As already anticipated, we use different arguments for the two cases $\kj = k$ and $\kj = k-1$. We will start by proving  \eqref{cellproj-u} and \eqref{cellproj-j} for $\kj=k$. We will then prove \eqref{cellproj-j} for $\kj=k-1$ by a similar argument which will not however yield, in this case, the desired bound \eqref{cellproj-u}. To prove the latter we will resort to an Aubin--Nitsche type argument. 

\begin{proof}[Proof of \eqref{cellproj-u} -- \eqref{cellproj-hu} for $\kj=k$]
	
In this case $\kappa=0$, $\epsB=0$, and $\Vh=\Sh$.
We use the stabilization-parametric HDG projector defined in
\eqref{proj1}--\eqref{proj3}, now with $\gamma=\tu$, and set
\begin{gather*}
\upi = \PiHs{\tu}(u,\fu), \qquad \hupi = \PiE{k}u, \qquad \fpiu = \PiHv{\tu}(u,\fu)
\\ U^\pi =
\PiHDG^{\tu}(U) =\bigl(\PiHs{\tu}(u,\fu),\PiE{k}u,
\PiHv{\tu}(u,\fu)\bigr).
\end{gather*}
Moreover, exactly as in Proposition \ref{prop:HDGproj} we have that, for $W = (w,\hw, \fw) \in \Vh$ 
	\begin{equation}\label{cancellation}	\Au(U - \PiHDG^{\tu}(U),W)	=    (\imu(\fu-\PiHv{\tu}(u,\fu)),\fw)_{\Omega}.\end{equation}

Under Assumption \ref{ass:smoothness}, using  \eqref{proj1}--\eqref{proj3}, with $\gamma=\tu$ gives
\begin{gather}
	\|u- \upi \|_{\O}
	\lesssim
	h^{\min\{s,k+1\}}\|u\|_{H^s(\O)}
	+h^{\min\{\jsmooth,k+1\}}\|\fu\|_{H^\jsmooth(\O)},
	\label{eq:HDGprojection-kk-u}\\
	\|\fu-\fpiu\|_{\O}
	\lesssim
h^{\min\{s,k+1\}}\|u\|_{H^s(\O)}
+h^{\min\{\jsmooth,k+1\}}\|\fu\|_{H^\jsmooth(\O)},
	\label{eq:HDGprojection-kk-j}
\end{gather}

{Moreover, using a triangle inequality  and \eqref{boundEhPro}, we have
\begin{multline}\label{eq:HDGprojection-kk-hu}
h^{1/2}	\| u^\pi - \hu^\pi \|_{\Eh} \leq h^{1/2} \| u^\pi - u \|_{\Eh} + h^{1/2}\| u - \hu^\pi \|_{\Eh} \\
	\lesssim h^{\min\{s,k+1\}}\|u\|_{H^s(\O)}
	+h^{\min\{\jsmooth,k+1\}}\|\fu\|_{H^\jsmooth(\O)}.
\end{multline}
}

We can bound
\[
I \leq M \| e \|_{\O} \| u - \upi \|_{\O} \leq \frac M 2 \| e \|_{\O}^2 + \frac M 2 \| u - \upi \|_{\O}^2.
\]		
Moreover it holds that
\begin{gather*}
| II | = | (\imu(\fu-\fu^\pi),\fe)_{\Omega} | \leq \frac \alpha 4 \| \fe \|_{\O}^2 + \alpha \| \fu-\fu^\pi \|_\O^2.
\end{gather*}		
To bound $III$, recalling that $\ConstPsi$ is defined as
\[
\ConstPsi=
|\inu\ctsens|\,\|\fp\|_{L^\infty(0,T;L^\infty(\O))},
\]
we have
\begin{equation*}
| III | = | c_{\fp}(u - \uIh,\fe) | \leq \frac {\ConstPsi^2} {\alpha} (2\| u - u^\pi \|_\O^2 + 2\| e \|_\O^2)+ 
\frac{\alpha}4 \| \fe \|_{\O}^2.
\end{equation*}
Combining the bounds for $I$, $II$ and $III$  we obtain
\begin{multline*}
	M \| e \|_{\O}^2 + | E |_{\Au}^2  \leq \left(
\frac	M 2    + \frac{2 \ConstPsi^2}{\alpha} \right) \| e \|_{\O}^2+
\frac{\alpha} 2 \| \fe \|_{\O}^2  \\+
	\frac M 2 \| u - \upi \|_{\O}^2  +
	\alpha \| \fu - \fu^\pi \|_{\O}^2  +
	\frac{2\ConstPsi^2}{\alpha} \| u - u^\pi \|_{\O}^2 
\end{multline*}

	Subtracting
\[
\left(
\frac	M 2 +   \frac{2\ConstPsi^2}{\alpha} \right) \| e \|_{\O}^2+
\frac{\alpha} 2 \| \fe \|_{\O}^2
\]
from both sides then yields
\begin{multline*}
	\left(
	\frac{M}{2}-\frac{2\ConstPsi^2}{\imu}
	\right)\| e \|_{\O}^2
	+\frac{\imu}{2}\|\fe \|_{\O}^2
	+\tu\|e-\he\|_{\Eh}^2
	+\Ku\|\he\|_{\GR}^2
\\	\le
		\left(
	\frac{M}{2} + \frac{2\ConstPsi^2}{\imu}
	\right)\| u - \upi \|_{\O}^2  +
\alpha \| \fu - \fu^\pi \|_{\O}^2.
\end{multline*}
Choosing
\(
M > M_0:=4\ConstPsi^2/\imu
\) implies that \(M/2 - 2\ConstPsi^2/\alpha\) is positive, which, combined with  \eqref{eq:HDGprojection-kk-u}, \eqref{eq:HDGprojection-kk-j} and \eqref{eq:HDGprojection-kk-hu}, by a triangle inequality, yields \eqref{cellproj-u}, \eqref{cellproj-j} and \eqref{cellproj-hu}.	
\end{proof}

\begin{proof}[Proof of \eqref{cellproj-j} and \eqref{cellproj-hu} for $\kj=k-1$]
In this case $\kappa=1$ and $\epsB$ may be positive.  
	Initially we proceed as for the case $k' = k$, using the $L^2$ projection instead of the stabilization-parametric HDG projection.
		We then let $U^\pi = (u^\pi,\hu^\pi,\fu^\pi) = \PiLdue(U) \in \Vh$ denote the $L^2$ orthogonal projection of  $U = (u,\hu,\fu)$.
	We have
	\begin{gather*}
		I = M (u^\pi-u,e)_\Omega = 0.
	\end{gather*}
We then focus on $II$ and $III$.	Thanks to Proposition~\ref{prop:L2} we have, for $\epsilon > 0$ to be chosen later,
	\begin{multline}\label{boundII}
		| II | =	| \Au(U - U^\pi,  \UIh - U^\pi) |
		\\
		\leq	|	(\tu h^{-1}\PiE{k-1}(u-u^\pi),  \PiE{k-1}( e - \he ) )_{\Eh}| + | ((\fu - \fu^\pi)\cdot \n , e - \he  )_{\Eh} |\\
		\le \tu h^{-1} \EHB{\| u  - u^\pi \|_{\Eh}} \| \PiE{k-1}( e - \he) \|_{\Eh} + \EHB{\| (\fu - \fu^\pi) \cdot \n \|_{\Eh}}\|  e - \he \|_{\Eh} \\
		\le \EHB{\tu {h^{-1}}  \| u - u^\pi \|_{\Eh}^2} + \frac{\tu h^{-1}}4 \| \PiE{k-1} (e - \he) \|_{\Eh}^2 + \EHB{\frac 1{4\epsilon}
		h \| (\fu - \fu^\pi) \cdot \n \|^2_{\Eh}} + \epsilon h^{-1} \| e - \he \|_{\Eh}^2.
	\end{multline}

As in \cite[Lemma 3.2]{qiu2016hdg}, since $\he \in \Poly{k-1}(\Eh)$, we  can bound $\| e - \widehat e \|_{\Eh}$ as 
	\begin{equation*}
		\| e - \widehat e \|^2_{\Eh} \leq 	\| \PiE{k-1} (e - \widehat e) \|^2_{\Eh} + C_\text{Poi} h \| \nabla e \|_{\Th}^2.
	\end{equation*}
	To bound $\| \nabla e \|_{\Th}$ we
observe that for $\fw \in (\Poly{k-1}(\Th))^2$, we have that
	\begin{multline*} (\nabla e,\fw)_\Th = - (e, \nabla \cdot \fw)_\Th + (e , \fw\cdot \n)_{\Eh} \\ = c_{\fp}(u^\pi - u,\fw)
	- \imu (\fe,\fw)_{\O} - (\PiE{k-1}(\he - e),\fw\cdot \n)_{\Eh} 	- c_{\fp}(e,\fw) .
	\end{multline*}
Setting $\fw = \nabla e \in (\Poly{k-1}(\Th))^2$ and applying multiple Cauchy--Schwarz and Young inequalities we obtain
	\begin{multline*}
		\| \nabla e \|_{\Th}^2 \leq | \alpha (\fe ,\nabla e)_{\Th} | + | ( \PiE{k-1} (\he - e),\nabla e \cdot \n)_{\Eh} | + | c_{\fp}(e,\nabla e) | +
		|	c_{\fp}(u^\pi - u,\nabla e)|\\
		\le
		\frac 1 2 \| \nabla e \|_{\Th}^2 + 2 C_\text{tr}^2 h^{-1} \| \PiE{k-1} (\he - e) \|^2_{\Eh}   + 2 \alpha^2 \| \fe \|_{\O} ^2 + 2 \ConstPsi^2 \| e \|_{\O}^2 + 2 \ConstPsi^2 \| u^\pi-u \|_{\O}^2,
	\end{multline*}
	whence, subtracting $\| \nabla e \|^2_{\Th}/2$ from both sides
	\[
	\frac12	\| \nabla e \|_{\O}^2 \leq 2 C_\text{tr}^2 h^{-1} \| \PiE{k-1} (\he - e) \|^2_{\Eh}   + 2 \alpha^2 \| \fe \|_{\O} ^2 + 2 \ConstPsi^2 \| e \|_{\O}^2 + 2 \ConstPsi^2 \| u^\pi-u \|_{\O}^2.
	\]
	Substituting in the bound for $\| e - \he \|_{\Eh}^2$ we obtain
	\begin{multline}\label{bound_e_unprojected}
h^{-1}\| e - \widehat e \|^2_{\Eh}
\leq
		(1 + 4 C_\text{Poi} C_\text{tr}^2) h^{-1}\|  \PiE{k-1} (e - \widehat e) \|^2_{\Eh} \\+
		4 C_\text{Poi} \alpha^2  \| \fe \|_{\O} ^2 + 4 C_\text{Poi}  \ConstPsi^2 \| e \|_{\O}^2 + 4 C_\text{Poi}  \ConstPsi^2 \| u^\pi-u \|_{\O}^2.
	\end{multline}
Then, with $\epsilon = \min\{
\tu/(4(1 + 4 C_\text{Poi}C^2_\text{tr})),1/(16 C_\text{Poi}\imu)
\}$, we can write
\begin{multline*}
|	II |
	\le \EHB{\tu {h^{-1}}  \| u - u^\pi \|_{\Eh}^2} + {\frac{\tu h^{-1}}4  \| \PiE{k-1} (e - \he) \|_{\Eh}^2 }+ \EHB{\frac 1{4\epsilon}
	h \| (	\fu - \fu^\pi) \cdot \n \|^2_{\Eh}} \\
	+ \left(
	{ \frac{\tu h^{-1}} 4 \|  \PiE{k-1} (e - \widehat e) \|^2_{\Eh} }+
	{ \frac{\imu}4 \| \fe \|_{\O} ^2} + { 4 \epsilon C_\text{Poi} \ConstPsi^2 \| e \|_{\O}^2} + 4 \epsilon C_\text{Poi}  \ConstPsi^2 \| u^\pi-u \|_{\O}^2
	\right).
\end{multline*}

	{
		Once again $\ConstPsi=|\inu\ctsens|\,\|\fp\|_{L^\infty(\O)}$.  Then
		\begin{multline*}
			|III|=|\Cpsi(u - \uIh,\fe)|
			\le \ConstPsi	\bigl(\| u - u^\pi \|_{\O} + \|e\|_\O	\bigr)
	\| \fe \|_{\O}\\
			\le 2 \frac{\ConstPsi^2}{\alpha} \| e \|_\O^2 + \frac \alpha 4 \| \fe \|_{\O}^2 + 2 \frac {\ConstPsi^2}{\alpha} \| u - u^\pi \|^2_{\O}.
	\end{multline*}}

	Assembling all the contributions and using the approximation bounds for the $L^2$ projections of Section \ref{sec:L2Proj} we obtain
		\begin{multline}\label{boundeI}
	\left(
	M - \frac{2 \ConstPsi^2}{\alpha} - 4 \epsilon C_\text{Poi} \ConstPsi^2
	\right)	\| e \|_{\O}^2
		+\frac{\imu}2\|\fe \|_{\O}^2
		+\frac{\tu h^{-1}} 2 \|\PiE{k-1} (e-\he)\|_{\Eh}^2
		+\Ku\|\he\|_{\GR}^2\\
		 \le
		  \EHB{\tu {h^{-1}}   \| u - u^\pi \|_{\Eh}^2}+ \left(\frac {2 \ConstPsi^2}{\alpha} + 4 \epsilon C_\text{Poi}  \ConstPsi^2
		 \right)
		   \| u - u^\pi \|_{\O}^2 + \EHB{\frac 1{4\epsilon}
		   	h \| (	\fu - \fu^\pi) \cdot \n \|^2_{\Eh}} \\
		\lesssim
		h^{2\min\{s,k+1\}-2} \| u \|^2_{H^s(\O)} + h^{2\min\{\sigma,k\}} \| \fu \|^2_{H^{\sigma}(\O)},
	\end{multline}
	which, provided $M$ is sufficiently large that the coefficient in front of $\| e \|_{\O}^2$ at the left hand side is positive, after a triangle inequality, gives us \eqref{cellproj-j}, as well as a first sub-optimal bound on $\| u - \uIh \|_{\O}$  ($\lesssim h^{\sbar-1}\|u\|_{H^s}+h^{\sigbar}\|\fu\|_{H^\jsmooth}$). The bound \eqref{cellproj-hu} follows by observing that
	\[
	\| \PiE{k-1} (\uIh - \huIh) \|_{\Eh} \leq 	\| \PiE{k-1} (e - \he) \|_{\Eh} + 	\EHB{\| \PiE{k-1} (u^\pi - u) \|_{\Eh}} + 	\underbrace{\| \PiE{k-1} (\hu^\pi - u) \|_{\Eh} }_{=0},
	\]
	which combined with \EHB{trace and approximation bounds} yields the desired result.
	\end{proof}

\begin{proof}[Proof of \eqref{cellproj-u}  for $k'=k-1$]
We now let  $\Zeta  = (\xi ,\hzeta ,\fzeta ) = (u - \uI, \hu - \huI, \fu - \fuI)$, and we consider the following adjoint problem: 
	find $\ve$ such that
	\begin{gather}\label{adjointStrong1}
		{- \nabla \cdot (\mu \nabla \ve) + M \ve
			+ \inu\chi\,\fp\cdot \nabla \ve = \xi ,} \\
		\label{adjointStrong2}	\mu \nabla \ve\cdot \n = 0, \text{ on $\GN$}, \qquad \mu \nabla \ve\cdot \n + \Ku \ve = 0, \text{ on $\GR$}.
	\end{gather}
For $M$ sufficiently large with respect to $\ConstPsi = | \beta \chi | \| \fp \|_{L^\infty(\O)}$, Problem \eqref{adjointStrong1}--\eqref{adjointStrong2} is well posed, and it is not difficult to see that, letting $\ve \in H^1(\O)$ denote its solution, $\hve = \gamma_{\Eh} \ve$ its single valued trace on $\Eh$ and
\[
\fve = \mu \nabla \ve,
\]
the triplet $\Ve = (\ve,\hve,\fve)$ satisfies,  for all $W = (w,\hw,\fw) \in \HTh{H^1(\Th)}\times L^2(\Eh) \times (\HTh{H^1(\Th)})^2$,
\begin{equation}\label{pb:adjoint}
	M(\ve,w) + \Au^*(\Ve,W) + \Cpsi^*(\fve,w) = (\xi ,w)_\Omega,
\end{equation}
	where
\begin{multline}
	\Au^*(\Ve,W) \\ = 
	(\imu \fve ,\fw)_{\Omega}
	+	  (\ve - \hve,\fw \cdot \n)_{\Eh} - 
	(\nabla \ve,\fw)_{\Th} 	- (
	\nabla \cdot \fve, w)_{\Th}  +  ( \fve\cdot \n,\hw)_{\Eh}\\ + \tu \, h^{-1}(\PiE{k-1}(\ve - \hve), \PiE{k-1}(w - \hw))_{\Eh}  +
	\Ku	(\hve, \hw)_{\GR} 
	= 	\Au(W,\Ve),
\end{multline}
and
\[
\Cpsi^*(\fv,w) = \Cpsi (w,\fv) = \C(w,\fp;\fv)  =
( \inu \ctsens   \fp \cdot \fv, w)_{\Th}.
\]
We observe that, for pure Neumann ($\GN = \partial\O$) or pure Robin ($\GR = \partial\O$) boundary conditions, $\xi  \in L^2(\O)$ implies that $\ve \in H^2(\O)$, with $\| \ve \|_{H^2(\O)} \lesssim \| \xi  \|_{\O}$. For mixed Neumann--Robin boundary conditions, we have instead that $\xi  \in L^2(\O)$ implies   that $\ve \in H^{2-\vareps}(\O)$ for all $\vareps$ with   $0<\varepsilon < 1/2$.  
Indeed,  $\ve$ can be seen as the solution of an elliptic problem with Neumann boundary conditions given by
\[
\nabla \ve\cdot \n = \begin{cases}
	0 & \text{on $\partial\Omega\setminus \GR$,}\\
	 - \frac{\Ku \ve}{\mu} & \text{on $\GR$.}
\end{cases}
\]
As,  for $\vareps$ with   $0<\varepsilon < 1/2$, $\nabla \ve \cdot \n \in H^{1/2-\varepsilon}(\partial\Omega)$, we then obtain (see \cite[Chapters 4-5]{grisvard2011elliptic}) that $\ve \in H^{2-\varepsilon}(\O)$ and
		\begin{equation}\label{smoothness_ve}
	\|\ve\|_{H^{2-\varepsilon}(\O)}\lesssim\|\xi \|_{\O}.
		\end{equation}
	Letting 
	$\Vpi = \PiLdue(\Ve) \in \Vh$, we then have
	\begin{multline*}
		\| \xi  \|_{\O}^2 = 	 M(\ve,\xi ) + \Au^*(\Ve,\Zeta ) + \Cpsi^*(\fve,\xi )  \\ =
		M(\xi ,\ve-\vpi) + \Au(\Zeta ,\Ve-\Vpi) + \Cpsi(\xi ,\fve-\fvpi). 
	\end{multline*}
Focusing, to fix the ideas, on the mixed boundary conditions case (the pure boundary condition case can be retrieved by setting $\vareps=0$), we have
	\[
	M(\xi ,\ve - \vpi) \leq M \| \xi  \|_{\O} \| \ve - \vpi \|_{\O} \lesssim M \| \xi  \|_{\O} h^{2-\varepsilon} \| \ve \|_{H^{2-\varepsilon}(\O)} {\color{magenta} \lesssim} M h^{2-\vareps} \| \xi  \|_\O^2.
	\]
	Moreover, by \eqref{A.26} we have
	\begin{multline*}
		\Au(\Zeta ,\Ve-\Vpi) = \Au(U - \PiLdue(U),\Ve - \Vpi) + \Au(\PiLdue(U) - \UIh,\Ve - \Vpi) \\
		= 	 \Au(U - \PiLdue(U),\Ve - \Vpi) 
		+ (\hu^\pi - \huIh - u^\pi + \uIh , \tu h^{-1}\PiE{k-1}(\ve- \vpi) - (\fve - \fvpi)\cdot \n )_{\Eh}.
	\end{multline*}
To bound the first term, we split
	\begin{equation*}
	\Au(U - \PiLdue(U),\Ve - \Vpi) = 
\Au(U - \PiLdue(U),\Ve) - \Au(U - \PiLdue(U),\Vpi).
	\end{equation*}
	Using \eqref{smoothness_ve}, as well as the approximation properties of $\PiLdue$, we have
	\begin{multline*}
	\Au(U - \PiLdue(U),\Ve) = (\xi ,u - u^\pi)_{\O} - M(\ve - \vpi,u-u^\pi)_\O - \Cpsi(u-u^\pi,\fve)
	\\ 
	\leq \| \xi  \|_\O \| u - u^\pi \|_\O + M \| \ve - \vpi \|_\O \| u - u^\pi \|_\O + \ConstPsi \| \fve \|_\O \| u - u^\pi \|_\O\\
	 \lesssim \| \xi  \|_\O h^{\sbar} \| u \|_{H^{s}(\O)} + M h^{2-\varepsilon} \| \ve \|_{H^{2-\varepsilon}(\O)} h^\sbar \| u \|_{H^{s}(\O)} +\ConstPsi \| \xi  \|_\O h^\sbar \| u \|_{H^{s}(\O)}.
	\end{multline*}
We also have
\begin{multline*}
	\Au(U - \PiLdue(U),\Vpi) = \tu h^{-1}(\PiE{k-1}(u-u^\pi),\PiE{k-1}(\vpi-\hvpi))_{\Eh} + ((\fu - \fu^\pi)\cdot\n,\vpi-\hvpi)_\Eh
\\
=\tu h^{-1}(\PiE{k-1}(u-u^\pi),\PiE{k-1}(\vpi-\ve))_{\Eh} + ((\fu - \fu^\pi)\cdot\n,\vpi-\ve)_\Eh {+ \underbrace{((\fu - \fu^\pi)\cdot\n,\ve-\hvpi)_\Eh}_{=(\fu\cdot\n-\PiE{k-1} (\fu \cdot \n), \ve-\hvpi)_{\partial\O}}}
\\
\lesssim  h^{-1} \EHB{\| \PiE{k-1}(u-u^\pi) \|_{\Eh} \| \PiE{k-1}(\vpi-\ve) \|_{\Eh}} + \EHB{\| (\fu - \fu^\pi)\cdot\n \|_{\Eh} \| \vpi-\ve \|_{\Eh}}\\ + {\| \fu\cdot\n-\PiE{k-1} (\fu \cdot \n)\|_{\partial\O} \| \ve-\hvpi \|_{\partial\O}}  \\
\lesssim h^{-1} h^{\min \{s,k+1\}-1/2} \| u \|_{H^s(\O)} h^{2-\varepsilon-1/2} \| \ve \|_{H^{2-\varepsilon}(\O)} +  h^{\min\{\sigma,k\}-1/2} \| \fu \|_{H^\sigma(\O)} h^{2-\varepsilon-1/2} \| \ve \|_{H^{2-\varepsilon}(\O)}
\\ + {h ^{\min\{\sigma-1/2,k\}} \| \fu \|_{H^{\sigma-1/2}(\partial\O)} h^{\min\{3/2-\varepsilon,k\}} \| \ve \|_{H^{3/2-\varepsilon}(\partial\O)}}
 \\
\lesssim \left(
h^{\min\{s,k+1\}-\varepsilon} \| u \|_{H^s(\O)} + h^{\min\{\sigma,k\} + 1 - \varepsilon} \| \fu \|_{H^\sigma(\O)}
\right) \| \xi  \|_{\O},
\end{multline*}
{where the last term of the second line reduces to a boundary term because $\ve - \hvpi$ is single valued on the mesh skeleton and orthogonal to $\Poly{k-1}(\Eh)$, $\fu^\pi\cdot\n\in\Poly{k-1}(e)$ on each edge $e$, and the normal component of $\fu\in H^\jsmooth(\O)^2$ is continuous across the interior edges (the analogous term in the first scalar product vanishes, since $\PiE{k-1}(\ve-\hvpi)=0$); on $\partial\O$ we have used $\fu\cdot\n\in H^{\jsmooth-1/2}(\partial\O)$ and $\ve\in H^{3/2-\varepsilon}(\partial\O)$. } For the case $k=1$ and $\kj=0$, the assumption $\jsmooth \ge 3/2-\varepsilon$ ensures that the bound remains valid.

The second term can be bounded as follows, where, in the last step, we use \eqref{bound_e_unprojected}, combined with \eqref{boundeI} and \eqref{standardL2}, which gives $h^{-1/2}\| e - \he \|_{\Eh} \lesssim h^{\sbar-1}\|u\|_{H^s(\O)} + h^{\sigbar}\|\fu\|_{H^\jsmooth(\O)}$:
\begin{multline}
	 | (\hu^\pi - \huIh - u^\pi + \uIh , \tu h^{-1}\PiE{k-1}(\ve- \vpi) - (\fve - \fvpi)\cdot \n )_{\Eh} | \\
	 \leq  h^{-1/2} \| e - \he \|_{\Eh} 
	 \left(\tu
h^{-1/2} \| \PiE{k-1} (\ve - \vpi) \|_{\Eh} + h^{1/2} \| \fve - \fvpi \|_{\Eh} 
	 \right) \\
	 \lesssim
 h^{-1/2} \| e - \he \|_{\Eh} \left(\tu
h^{-1/2} h^{2-\varepsilon-1/2}\| \ve \|_{H^{2-\varepsilon}(\O)} + h^{1/2} h^{1-\varepsilon-1/2} \| \fve \|_{H^{1-\varepsilon}(\O)}  
\right) \\
\lesssim h^{1-\varepsilon}\left(
h^{\min\{s,k+1\} - 1} \| u \|_{H^s(\O)} + h^{\min\{\sigma,k\}} \| \fu \|_{H^\sigma(\O)}
\right) \| \xi  \|_{\O}.
\end{multline}

Moreover, it holds that
	\begin{multline*} | \Cpsi(\xi ,\fve-\fvpi) | = (\beta \ctsens  \fp \cdot(\fve - \fvpi),\xi )_{\O} \leq \| \beta \ctsens  \fp \|_{L^\infty(\O)} \| \fve - \fvpi\|_{\O} \| \xi  \|_{\O} \\
		\lesssim  h^{1-\vareps}\, \| \fve \|_{H^{1-\vareps}(\O)} \| \xi  \|_{\O} \lesssim  h^{1-\vareps} \| \xi  \|_\O^2,
	\end{multline*}
	and, assembling all the contributions, we finally get
	\begin{equation*}
		\| \xi  \|_{\O}^2 \leq C_0 h^{2-\vareps} \| \xi  \|^2_{\O}  +
		C_1  \| \xi  \|_\O 
		(h^{\sbar - \vareps} \| u \|_{H^{s}(\O)} + h^{\sigbar +1-\varepsilon} \| \fu \|_{H^\jsmooth (\O)} )
		+ C_2  h^{1-\vareps}
		\| \xi  \|^2_{\O} 
	\end{equation*}
where $C_0$ depends on $M$,	whence, dividing by $\| \xi  \|_{\O}$,
	\[
	(1-C_0 h^{2-\vareps}- C_2  h^{1-\vareps}) \| \xi  \|_{\O} \lesssim h^{\sbar - \vareps} \| u \|_{H^{s}(\O)} + h^{\sigbar +1-\varepsilon} \| \fu \|_{H^\jsmooth (\O)} .
	\]
	Provided $h$ is small enough that $1-C_0 h^{2-\vareps}- C_2  h^{1-\vareps} \geq 1/2$, we get the desired bound.
\end{proof}

\newcommand{\fqh}{\boldsymbol{q}_h}

	\begin{proof}[Proof of \eqref{cellproj-dt} for both $\kj = k$ and $\kj = k-1$]
		Introduce the shifted bilinear form
		\[
		\mathcal B_{\fp}(X,W)
		:=M(x,w)_\O+\Au(X,W)+\C(x,\fp;\fw),
		\qquad X=(x,\widehat x,\boldsymbol x).
		\]
		For almost every $t\in(0,T)$, let $\widetilde U_h^I$ be the elliptic
		projection of $\dt U$ with the value $\fp(t)$ fixed:
		\[
		\mathcal B_{\fp}(\widetilde U_h^I,W)
		=\mathcal B_{\fp}(\dt U,W)
		\qquad\forall W\in\Vh.
		\]
		Differentiating \eqref{pb:elliptic} in time gives
		\[
		\mathcal B_{\fp}(\dt\UIh,W)
		+\C(\uIh,\dt\fp;\fw)
		=\mathcal B_{\fp}(\dt U,W)
		+\C(u,\dt\fp;\fw).
		\]
		Consequently, for
		\[
		Q_h = (q_h,\widehat q_h,\fqh ):=\dt\UIh-\widetilde U_h^I,
		\]
		we have the exact identity
		\begin{equation}\label{eq:time-correction-elliptic-km1k}
			\mathcal B_{\fp}(Q_h,W)
			=\C(u-\uIh,\dt\fp;\fw)
			\qquad\forall W = (w,\hw,\fw)\in\Vh.
		\end{equation}
		Testing with $W=Q_h$, using the coercivity of the shifted form and Young's 
		inequality, eventually yields
		\[
		\|q_h\|^2_\O
		+\|\fqh\|^2_\O
		+h^{-\kappa}\| \PiE{\kj}(q_h -\widehat q_h)\|^2_{\Eh}
		\leq C
		\|\dt\fp\|^2_{L^\infty(0,T;L^\infty(\O))}\|u-\uIh\|^2_\O + \frac12 \|  \fqh\|_{\O}^2,
		\]
		which, by subtracting $ \|  \fqh\|_{\O}^2/2$ from both sides results in
		\[
			\|q_h\|^2_\O
		+ \frac12 \|\fqh\|^2_\O
		+h^{-\kappa}\| \PiE{\kj}(q_h -\widehat q_h)\|^2_{\Eh}
		\lesssim
		\|\dt\fp\|^2_{L^\infty(0,T;L^\infty(\O))}\|u-\uIh\|^2_\O.
		\]				
		On the other hand, under Assumption \ref{ass:smoothness}, which provides the necessary regularity for the time derivative, ensuring in particular that $\dt\fu\cdot\n$ is continuous across interfaces, we can apply \eqref{cellproj-u} to $\dt U$, which gives the
		corresponding bound for $\dt u-\widetilde u_h^I$.  The identity
		\[
		\dt U-\dt\UIh=(\dt U-\widetilde U_h^I)-Q_h
		\]
		and the pointwise estimate \eqref{cellproj-u} for $u-\uIh$ therefore give
		\eqref{cellproj-dt}.
	\end{proof}

\smallskip


\section{Numerical tests}\label{sec:numerical}

In this section, we present numerical experiments to test the performance of the proposed HDG methods.
We consider two test cases: 
the first is a problem with known exact solution, used to verify the convergence rates of the method,  while the second is a more challenging scenario exhibiting blow-up behavior.

We test the formulations introduced in Section~\ref{sec:HDGU-scheme} for
polynomial degrees $k=1,2$. 
The notation HDG$(k',k)$ indicates that the cellular flux $\fu$
and the trace $\widehat{u}_h$ are approximated by polynomials of degree $k'$,
while the cellular density $u$, the chemoattractant $\p$, its trace
$\widehat{\varphi}_h$, and the chemical flux $\fp$ are
approximated by polynomials of degree $k$.
Accordingly, the configurations considered are
\begin{center}
	\HDGAzu ,\,
	\HDGAud,\,
	\HDGku,\,
	\HDGkd .
\end{center}
For simplicity, time derivatives are discretized by the implicit Euler method.
Unless otherwise stated, in all simulations the nonlinear problem is solved by
Picard iteration with tolerance $10^{-12}$, and the stabilization parameters
are set to $\gamma_u=\gamma_\varphi=1$.
All numerical experiments were performed in \textsc{Matlab} R2025a
on an iMac equipped with a 3.6 GHz 10-core Intel Core i9 processor and
128 GB of RAM.

\begin{remark}
	We briefly comment on a few issues related to the implementation of the hybridized discontinuous Galerkin method used in the experiments. 
	The efficient implementation of the method relies on the static condensation of the element-level variables: solving, element by element, a local problem reduces the global discrete problem to a system involving only the trace variables $\hu_h$ and $\hp_h$. As a result, only the trace unknowns enter the global linear system 	that can be assembled with a storage cost comparable to that of conforming methods of the same order, while retaining the advantages of discontinuous Galerkin methods. 

	In the case of the Keller--Segel system the local problems are nonlinear,
	because of the chemotactic term $\chi\,u\,\nabla\varphi$. A Newton method
	would require the derivation and implementation of the Jacobian of the
	discrete nonlinear operator, including its contribution to the statically
	condensed system; moreover, any change in the model (additional equations,
	different chemotactic sensitivities or reaction terms) would require the
	Jacobian to be derived and implemented anew. We therefore resort to a
	fixed-point (Picard) iteration: at each iteration the nonlinear terms are
	linearized by freezing the chemical flux $\fp_h$ at the previous
	iterate, so that each step only requires the solution of a linear HDG problem
	of the same structure as in the linear case, with the same local solvers and
	the same static condensation. 
    In our experiments, 
    the iteration
	converges within a few iterations per time step
	(3 in Example~1 and between 5 and 10 in Example~2).

	This choice is also motivated by the  prospect 
    of extending the method to
	cancer-on-chip models, where additional equations and different
	nonlinearities must be taken into account; in this setting, the Picard
	approach provides a flexible strategy, since new nonlinear terms can be
	handled by the same freezing procedure without requiring the derivation of new Jacobians.
\end{remark}

\subsection{Example 1}
\label{subsec:manufactured-solution}

We begin with a test case designed to verify the convergence properties. %

The computational domain is the unit square $\Omega=[0,1]^2$, with
non-homogeneous Neumann boundary conditions imposed on the entire boundary
$\partial\Omega$.  The exact solutions are
\[
u(x,y,t) = \cos(t)\sin(\pi x)\cos(\pi y), \qquad
\varphi(x,y,t) = \sin(t)\cos(\pi x)\sin(\pi y),
\]
the source terms $f_u$, $f_\varphi$ in equations~\eqref{KS1} and~\eqref{KS2} and
the Neumann data on $\partial\Omega$ are computed accordingly.

We fix the final time $T=1$ and the physical parameters as $\mu = \nu = a = b = 1$ and $\chi = 1$.
We investigate the spatial convergence on a sequence of unstructured
triangular meshes generated with the \textsc{Matlab} mesh generator \texttt{initmesh}
(PDE Toolbox), with maximum edge length
$h \in \{2^{-1}, 2^{-2}, 2^{-3}, 2^{-4}\}$.  
The time step is kept fixed to $dt = 2\times 10^{-4}$; halving it leaves
the errors on the finest mesh essentially unchanged, confirming that the
temporal error is negligible with respect to the spatial one.

%

Figures~\ref{fig:hdg-h-history} and~\ref{fig:hdg-h-PkPk-history} present the
$L^2(\Omega)$-errors at the final time $T=1$ for HDG$(k-1,k)$ and HDG$(k,k)$,
$k=1,2$, respectively, together with the theoretical convergence rates,
represented by dashed and dotted lines.
\begin{figure}[htbp]
	\centering
		 \includegraphics[width=0.45\textwidth]{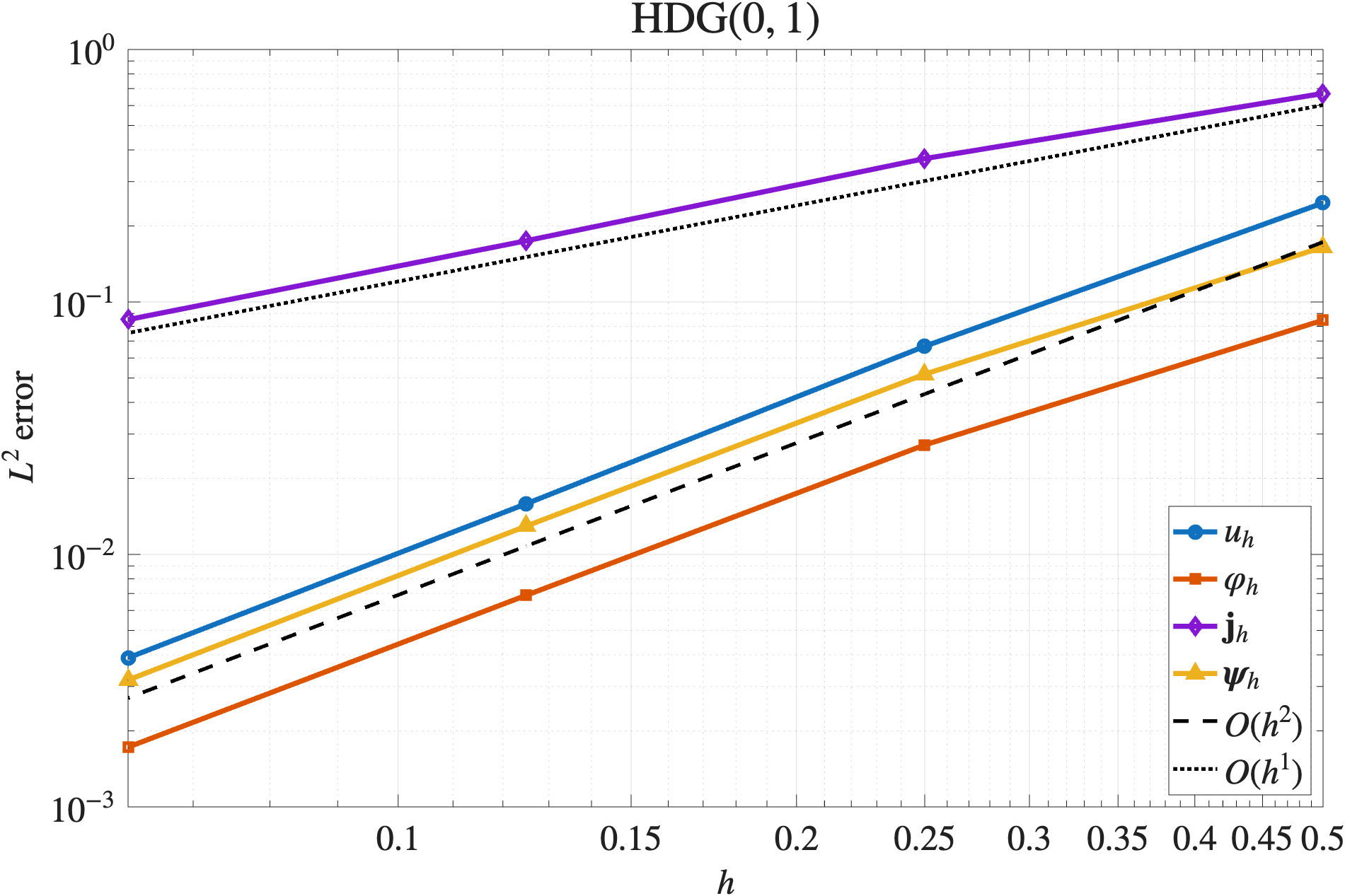}  
	 \includegraphics[width=0.45\textwidth]{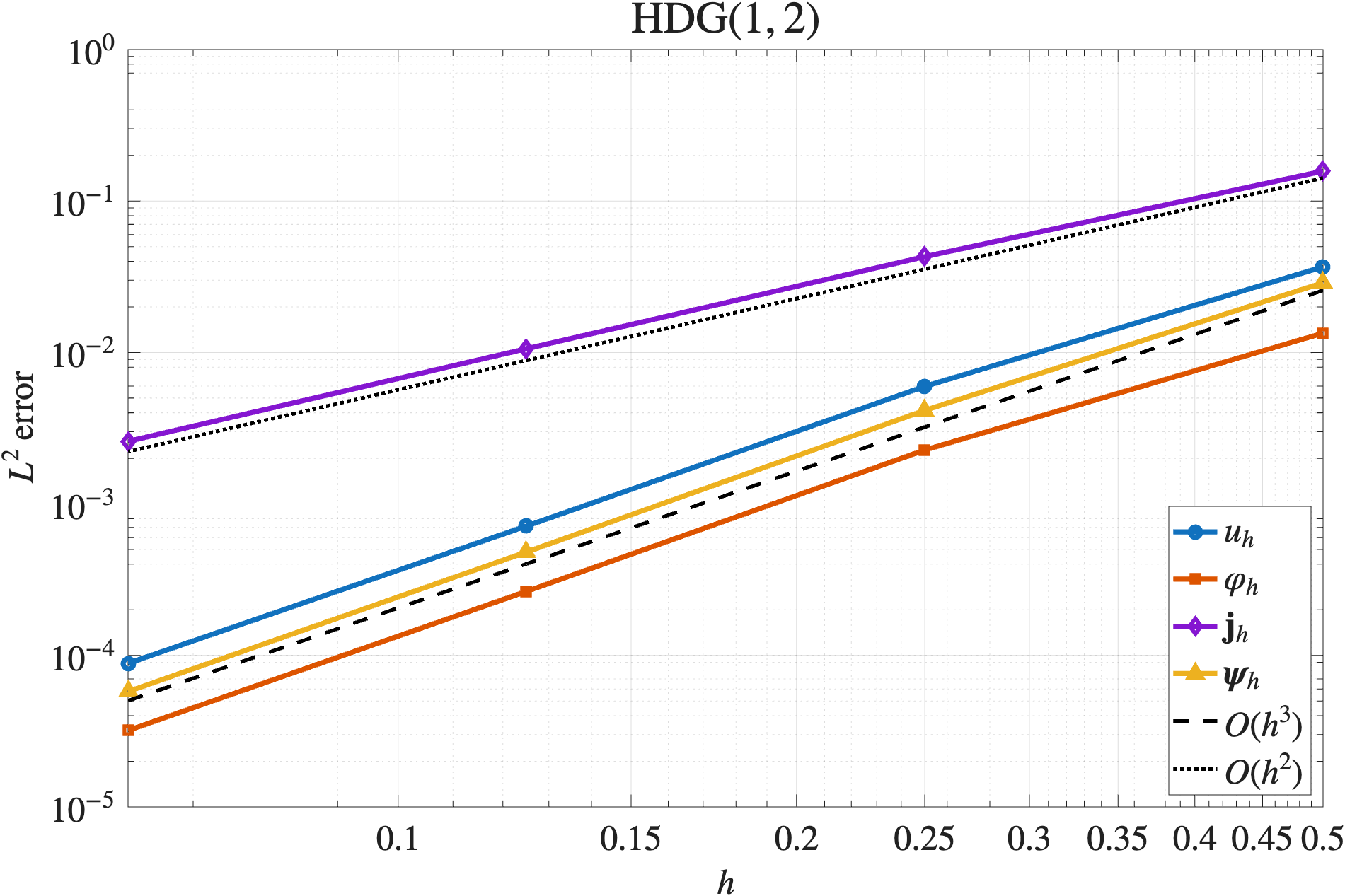}  
	\caption{Example 1: $L^2(\Omega)$-errors of $u_h$, $\varphi_h$, $\fu_h$
		and $\fp_h$ at the final time $T=1$, as functions of the
		characteristic mesh size $h$, for HDG$(0,1)$ (left) and HDG$(1,2)$ (right).
		Dashed and dotted lines indicate the reference slopes $h^{k+1}$ and $h^{k}$,
		respectively.}
	\label{fig:hdg-h-history}
\end{figure}

\begin{figure}[htbp]
	\centering
	\includegraphics[width=0.45\textwidth]{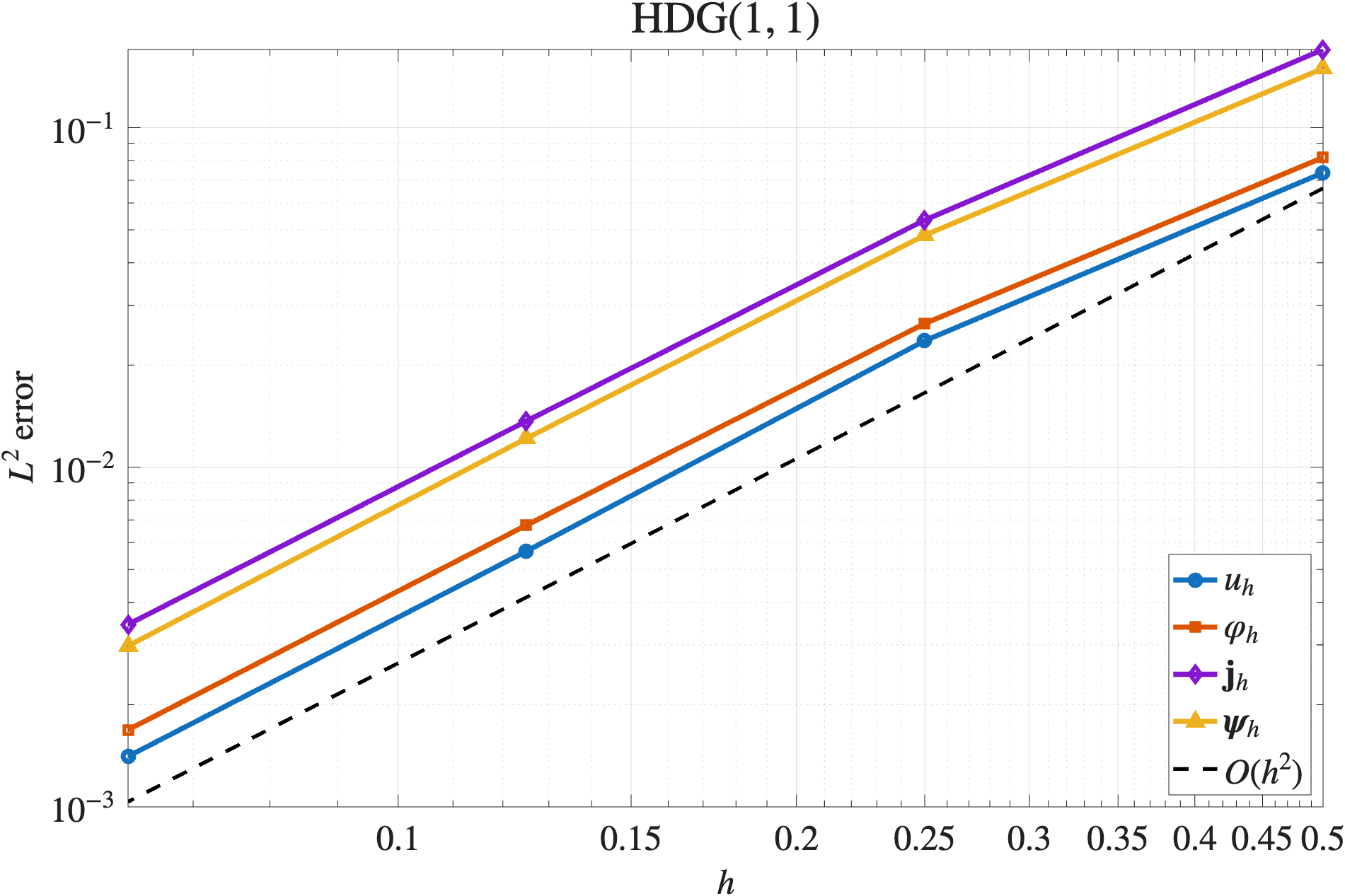}  
	\includegraphics[width=0.45\textwidth]{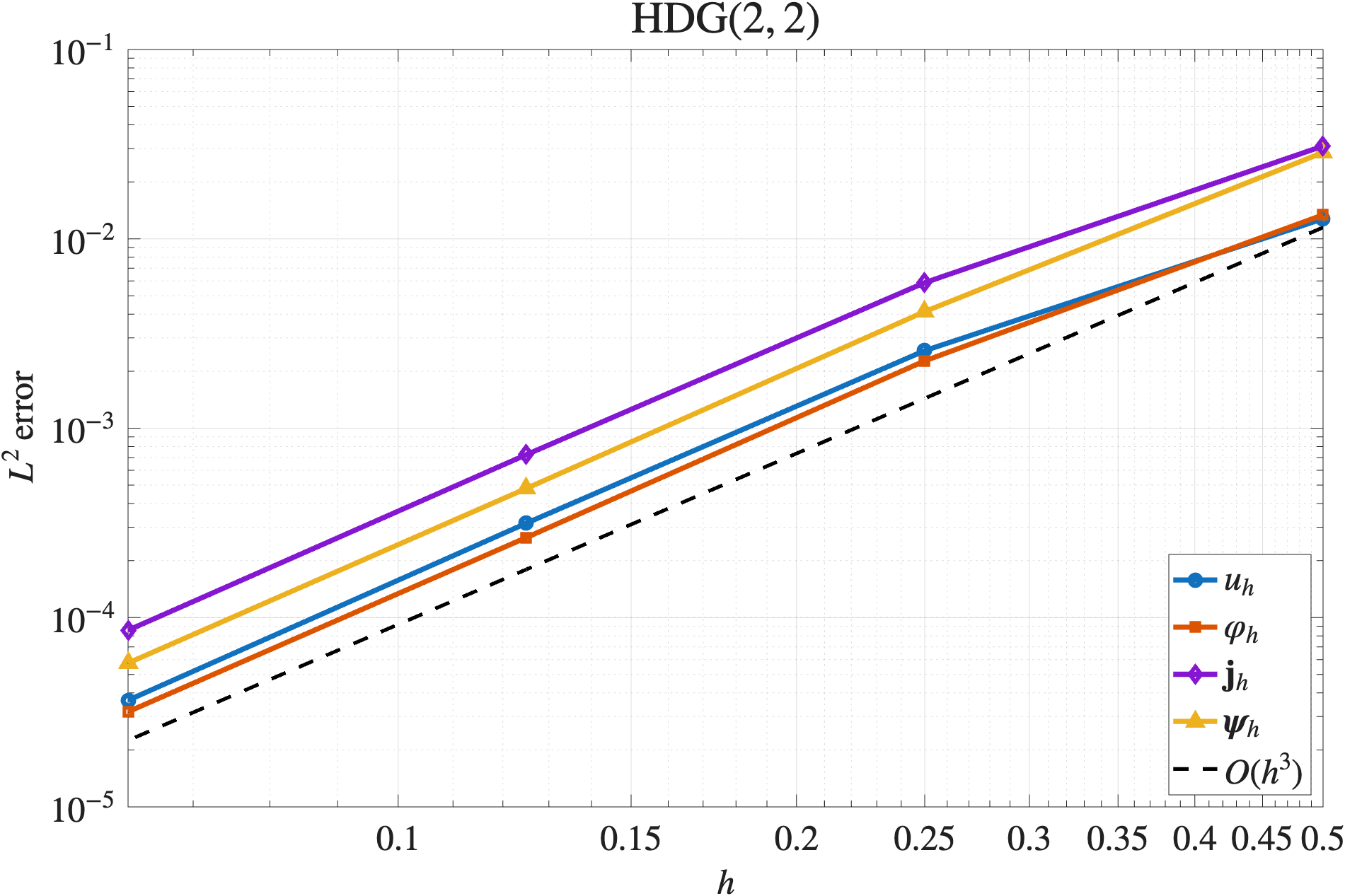}  
\caption{Example 1: $L^2(\Omega)$-errors of $u_h$, $\varphi_h$, $\fu_h$
	and $\fp_h$ at the final time $T=1$, as functions of the
	characteristic mesh size $h$, for HDG$(1,1)$ (left) and HDG$(2,2)$ (right).
	The dashed line indicates the reference slope $h^{k+1}$.}
\label{fig:hdg-h-PkPk-history}
\end{figure}

The results agree with Theorem~\ref{thm:KSerror} (see Table~\ref{table:HDG}): 
the scalar variables $u_h$ and $\varphi_h$, as well as the flux $\fp_h$, converge with the optimal order $\mathcal{O}(h^{k+1})$, while the flux $\boldsymbol{j}_h$ converges with order $\mathcal{O}(h^{k})$ in the case \HDGA,  
for $k=1,2$.
For the case \HDGk, 
with $k=1,2$, optimal convergence of order $\mathcal{O}(h^{k+1})$ is observed for all variables, as expected.

\subsection{Example 2: Blow-up dynamics with chemotaxis}
\label{subsec:blowup-linear}

After validating the theoretical convergence rates in smooth regimes, we
investigate the ability of the proposed HDG scheme to reproduce the strong
cell aggregation associated with finite-time blow-up in the Keller--Segel
model. This experiment is intended to assess the qualitative robustness of
the method in a strongly aggregation-dominated regime and does not constitute
a numerical analysis of blow-up formation.

The computational domain is the square
\[
\Omega=\left[-\tfrac{1}{2},\tfrac{1}{2}\right]^2\subset\mathbb{R}^2,
\]
and homogeneous Neumann boundary conditions are imposed on
$\partial\Omega$. The initial conditions are highly concentrated Gaussian
profiles centered at the origin:
\begin{equation}
\label{eq:initial-conditions-ex2}
u_0(x,y)=840\exp\bigl(-84(x^2+y^2)\bigr),
\qquad
\varphi_0(x,y)=420\exp\bigl(-42(x^2+y^2)\bigr).
\end{equation}
The physical parameters are the same as in Example~1.

We consider the time interval $[0,T]$, with $T=1.2\times10^{-4}$.
The initial cell density has total mass
$m(0)=\int_\Omega u_0\,dx\approx 10\pi$, which exceeds the critical value $8\pi$ for the present choice of
parameters \cite{horstmann_1970_2003}, so that the cell density is expected
to concentrate at the origin and to blow up in finite time as in 
\cite{epshteyn_discontinuous_2009}. 
This regime is particularly challenging from the numerical point of view,
since resolving an increasingly localized and steep density profile requires
very fine spatial and temporal discretizations. Moreover, the dominant
chemotactic transport may generate spurious oscillations, including
nonphysical negative concentrations.

Figure~\ref{fig:unstabilized} 
shows the numerical solution obtained with the HDG(2,2) scheme on a
mesh with maximum edge length $h = 1/70$, $dt = 10^{-5}$ and
$\gamma_u=\gamma_\varphi=1$. 
The method fails to reproduce the expected
qualitative behavior: spurious oscillations develop 
and negative values of the cell density appear at $t = 6\times 10^{-5}$,
where the simulation is stopped.

\begin{figure}[htbp]
\centering
\includegraphics[width=.45\textwidth]
{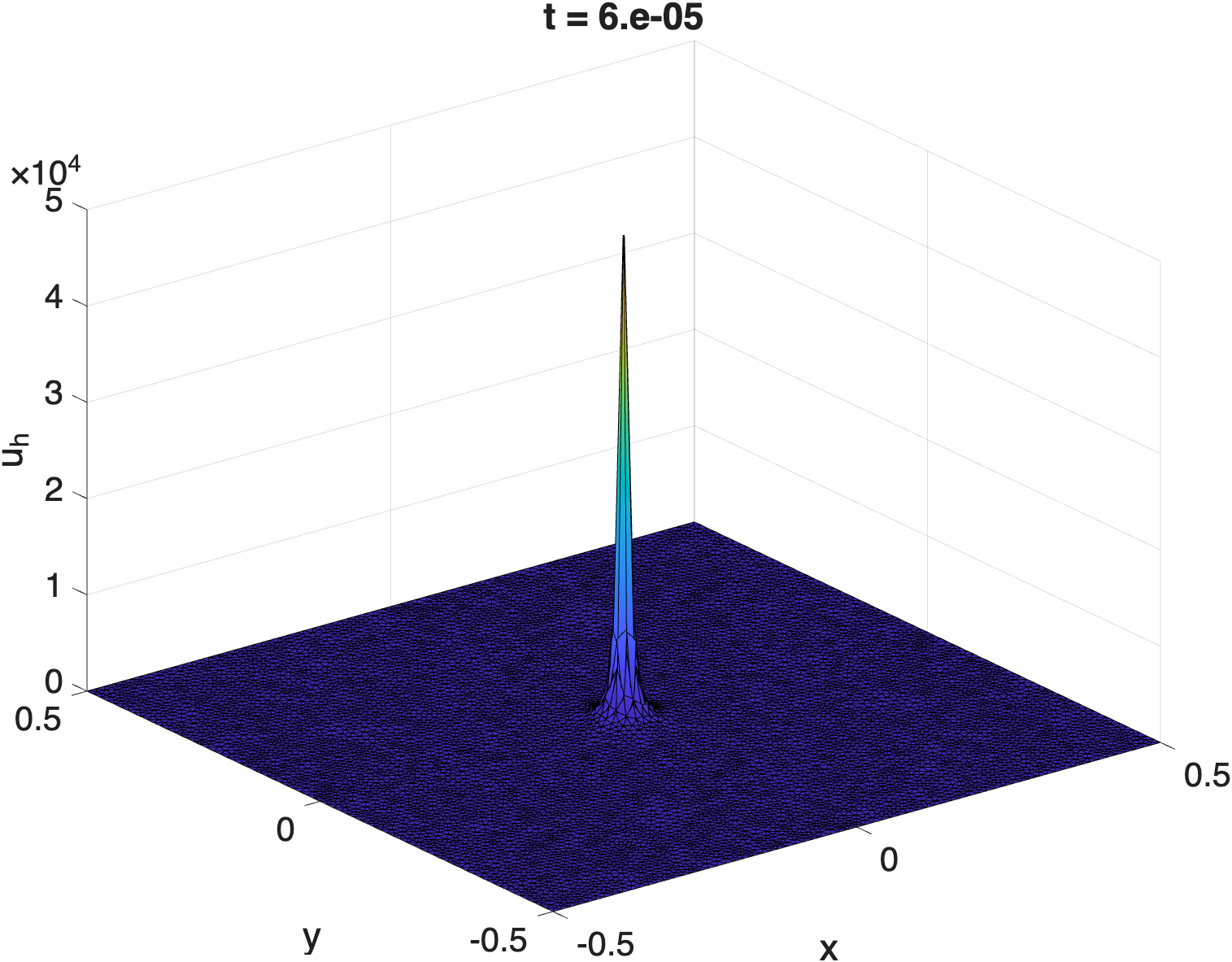}
\includegraphics[width=.45\textwidth]
{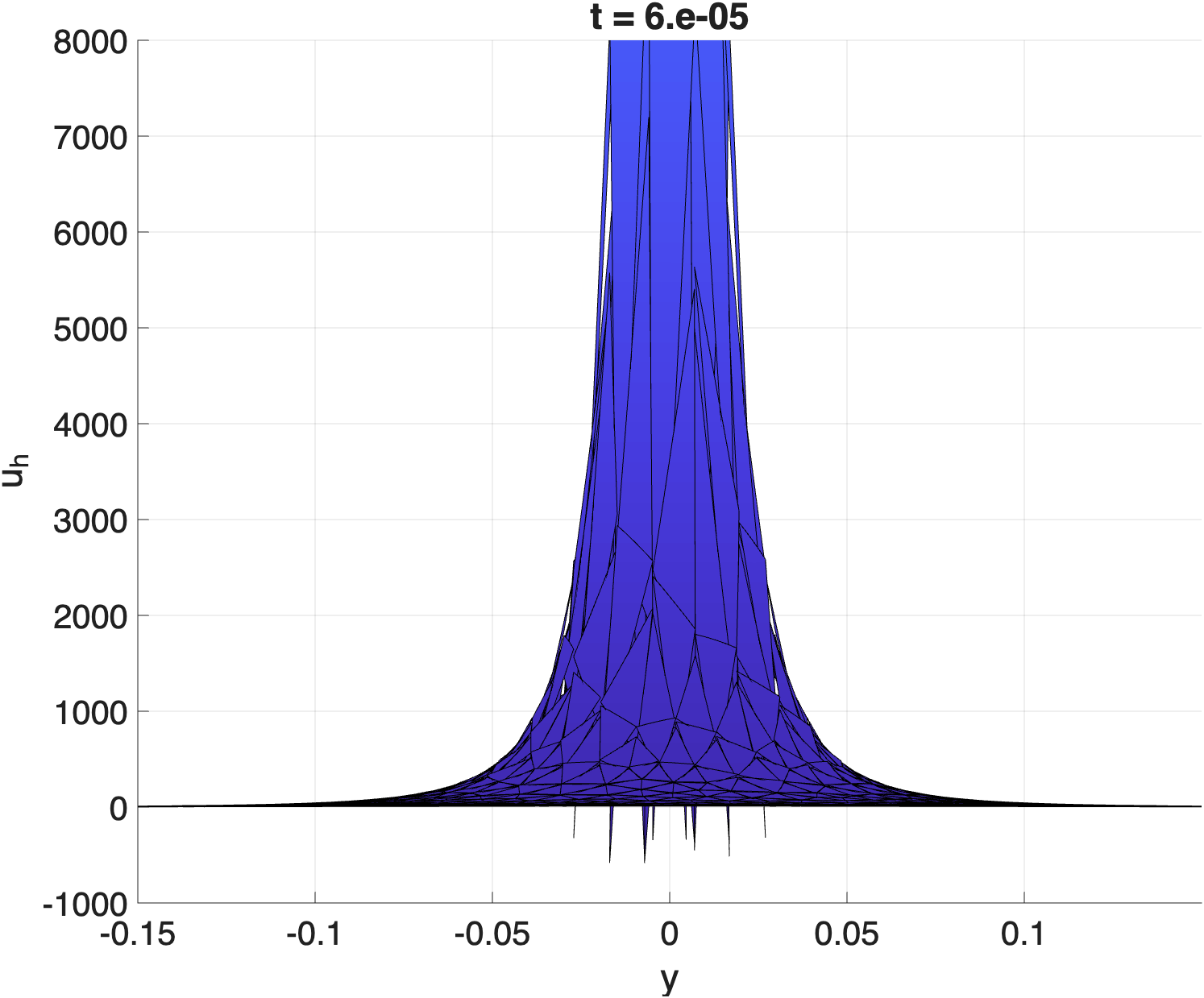}
\caption{Example 2: HDG$(2,2)$ without upwind stabilization, $h=1/70$,
$dt=10^{-5}$, at $t=6\times10^{-5}$, when negative values of $u_h$
appear and the simulation is stopped. Left: 3D view; right: zoomed side
view in the $(y,u_h)$-plane.}
\label{fig:unstabilized}
\end{figure}

To control the observed oscillations and 
the loss of positivity,
we modify the HDG scheme by introducing an upwind stabilization in
the numerical normal flux on $\partial K$:
\begin{equation}
\label{eq:numerical_flux}
\widehat{\boldsymbol j}_h\cdot\boldsymbol n_K
=
\boldsymbol j_h\cdot\boldsymbol n_K
+\gamma_u\bigl(u_h-\widehat u_h\bigr),
\end{equation}
where $\boldsymbol n_K$ denotes the outward unit normal to $\partial K$.
 As in  \cite{CockburnDong2009}, where the total flux is likewise taken as an independent unknown, we split the stabilization parameter as 
\[
\gamma_u=\gamma_D+\gamma_C,
\]
where $\gamma_D>0$ is the diffusive contribution and $\gamma_C$ is the upwind
contribution associated with chemotactic transport, acting on the inflow part of $\partial K$; 
see also \cite{Nguyen2009Linear,Nguyen2009Nonlinear,EggerHDG1010}  for the case in which the convective flux is treated separately.

Recalling that $\fp=-\nu\nabla\varphi$, the chemotactic velocity can be written as
\[
\boldsymbol b_{\rm chem}=\chi\nabla\p=-\beta\mu\widetilde\chi\,\fp,
\]
so that the inflow condition $\boldsymbol b_{\rm chem}\cdot\boldsymbol n_K<0$ is
equivalent to $\fp\cdot\boldsymbol{n}_K>0$. 
Accordingly, the convective
stabilization parameter is evaluated edgewise as
\begin{equation}
\label{gamma_upwind}\gamma_C|_e = \max\Big(\beta\mu \widetilde\chi\,\frac{1}{|e|}\int_e \fp_h\cdot\boldsymbol{n}_K,\,0\Big),
\qquad e\subset\partial K .
\end{equation}

We consider the upwind-stabilized HDG(2,2) scheme on the same mesh and with the same time step, and set $\gamma_D = 1$.
In  Figure~\ref{fig:P2P2_h70_u_phi} we report snapshots
of the computed solution $u_h$ at selected time instants
$t=0$, $t=6\times10^{-5}$, and $t=1.2\times10^{-4}$. 
The bottom-right panel displays a side view of $u_h$ in the $(y, u_h)$-plane at
$t=1.2\times10^{-4}$.

\begin{figure}[htbp]
\centering%
\includegraphics[width=0.45\textwidth]{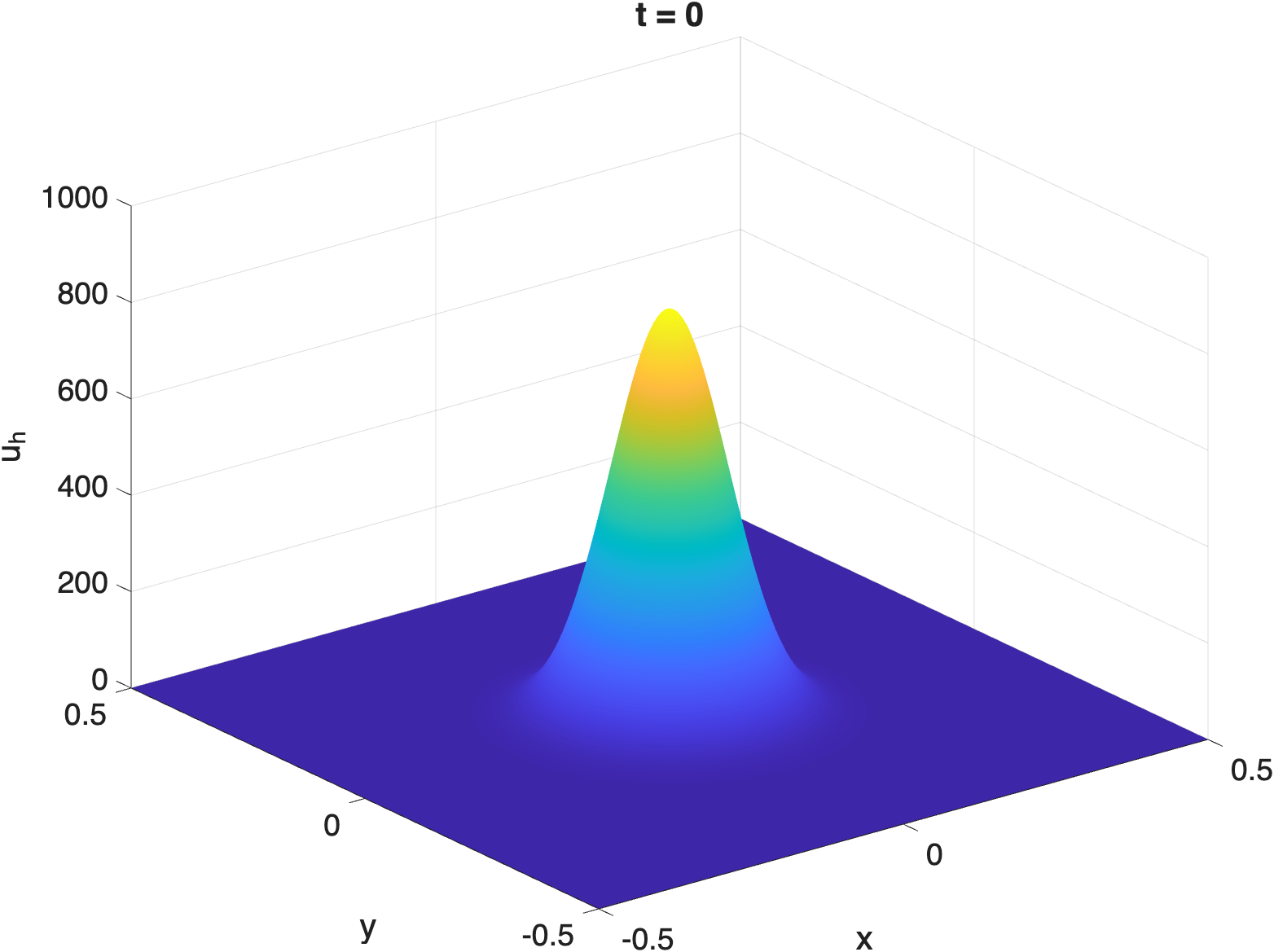}%
\includegraphics[width=0.45\textwidth]{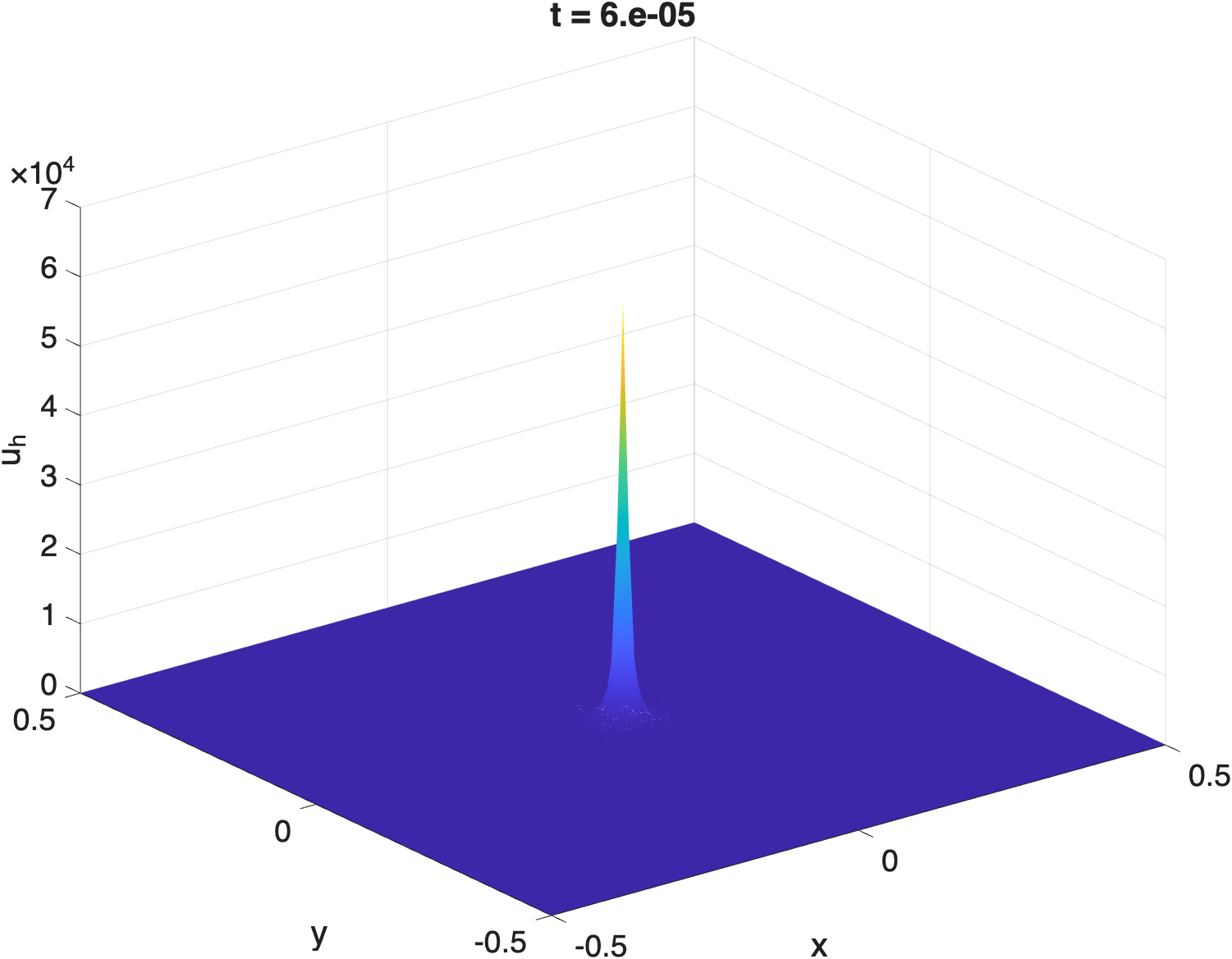}  
\includegraphics[width=0.45\textwidth]{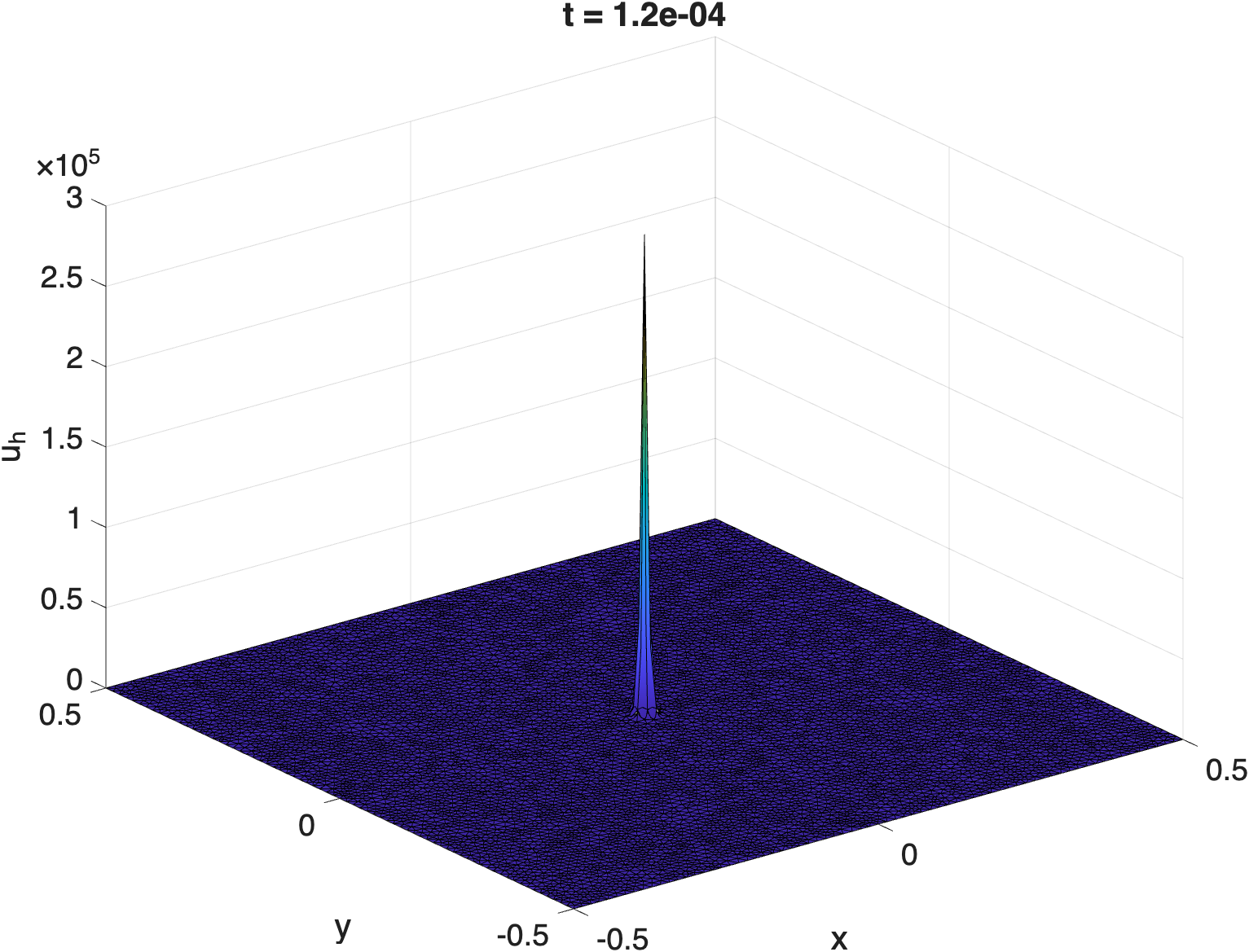} 
\includegraphics[width=0.4\textwidth]{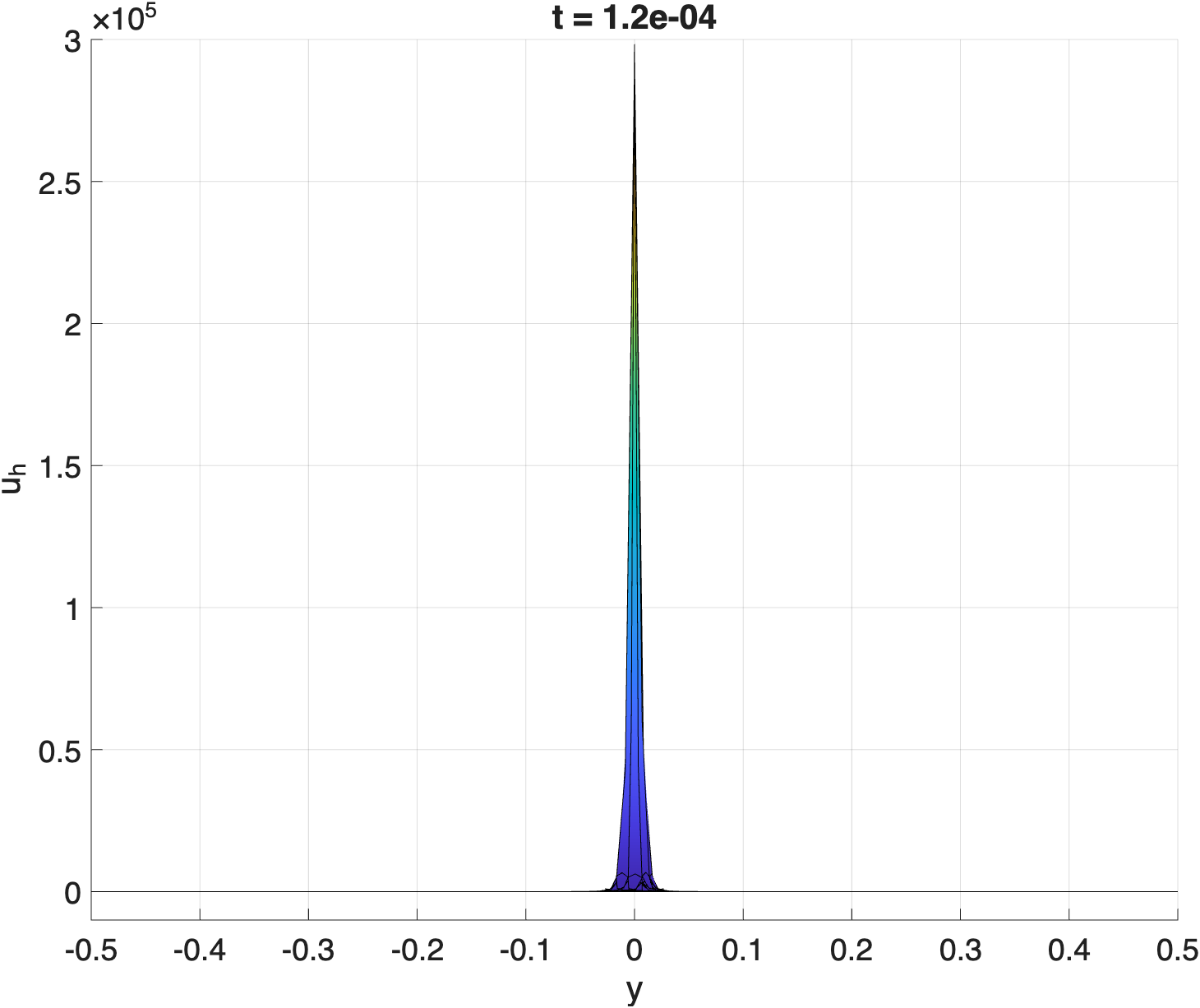} 
\caption{Example 2: cell density $u_h$ computed with the upwind-stabilized
	HDG$(2,2)$ scheme, $h=1/70$ and $dt=10^{-5}$, at $t=0$ and
	$t=6\times10^{-5}$ (top), and at $t=1.2\times10^{-4}$ (bottom: three-dimensional
	view, left; side view in the $(y,u_h)$-plane, right).}
\label{fig:P2P2_h70_u_phi}
\end{figure}

We observe that the cell density rapidly increases over time, with a
progressive concentration of mass toward the origin, as expected.
The total mass $m(t)=\int_\Omega u_h(t)\,dx$ is preserved up to round-off:
$\max_n |m(t_n)-m(0)|/m(0) \approx 6\times 10^{-16}$, and the values of
$u_h$ at the Lagrange nodes remain positive at all time steps.
The same holds when the simulation is repeated at fixed $h = 1/70$ with
$dt/2$ and $ dt/4$.

These preliminary results indicate that the proposed HDG method can
reproduce the expected qualitative behavior in this challenging blow-up
regime, while preserving the total mass and the nodal positivity of $u_h$.
A more detailed investigation of tailored upwind
stabilization strategies, of improved time discretization schemes and of positivity preserving limiters  is
left for future work.

%

 \section{Conclusions}\label{sec:conclusions}
We proposed and analyzed two hybridizable discontinuous Galerkin methods for the Keller--Segel system, which differ in the polynomial degree used for the cellular flux. Taking the total flux, including the chemotactic drift, as an independent unknown allows for an equal-order method that converges with order $k+1$  in all the variables, with a proof that relies on an energy argument only and requires neither the convexity of the domain nor a duality argument. For the method with a lower order flux we proved order $k$ for the cellular
flux and order $k+1$ for all the other variables, up to an arbitrarily small
loss in the case of mixed Neumann--Robin boundary conditions.
The numerical tests confirm these rates.

Future work will be devoted to
the extension of the method to more complex chemotaxis models, such as, e.g., the
one describing demyelination in multiple sclerosis, \cite{lombardo2017demyelination}, 
and to the
coupled models arising in the simulation of cancer-on-chip experiments.

\appendix

\section{Properties of the projections}\label{section:prop_projections}
We recall here some properties of the projections introduced in Section~\ref{sec:proof}. 

\subsection{Properties of the HDG projection}\label{sec:approximationHDG}

\begin{proof}[Proof of Proposition \ref{prop:HDGproj}] 
	Letting $\widehat \phi^\pi = \PiE{k}\phi$, we can write 
	\begin{multline*}
		\Ap(\Phi - \PiHDG^{\tp}(\Phi),\Theta)  = 
		\underbrace{((\fp - \PiHv{\tp}(\p,\fp))  \cdot\n, \t)_{\Eh}  
			+ \tp(\p - \PiHs{\tp}(\p,\fp),\t)_{\Eh} }_{=0}
		- \underbrace{\tp(\hp - \hp^\pi,\t)_{\Eh} }_{=0}\\ -
		\underbrace{	(  \fp - \PiHv{\tp}(\p,\fp),\nabla\t)_{\Th}}_{=0}  + (a (\p-\PiHs{\tp}(\p,\fp)),\t)_\Omega +  (\inu (\fp-\PiHv{\tp}(\p,\fp)),\ft)_{\Omega}	\\ - \underbrace{(\p-\PiHs{\tp}(\p,\fp) ,
			\nabla \cdot \ft)_{\Th}}_{=0}
		+ \underbrace{(\hp -\hp^\pi , \ft\cdot \n)_{\Eh} }_{=0}\\
		\underbrace{-  ( (\fp - \PiHv{\tp}(\p,\fp)) \cdot \n ,\ct)_{\Eh} - \tp(\p-\PiHs{\tp}(\p,\fp),\ct)_{\Eh}}_{=0} + \underbrace{\tp(\hp - \hp^\pi,\ct)_{\Eh}}_{=0} +  \underbrace{\Kp (\hp - \hp^\pi, \ct  )_{\GR}}_{=0}\\
		=  (a (\p-\PiHs{\tp}(\p,\fp)),\t)_\Omega +  (\inu (\fp-\PiHv{\tp}(\p,\fp)),\ft)_{\Omega},
	\end{multline*}
	which  concludes the proof.
\end{proof}

\begin{proof}[Proof of Proposition \ref{prop:L2}]
	Letting $W^\pi = (w^\pi,\hw^\pi,\fw^\pi) = \PiLdue(W) $,	and using the definition \eqref{numflux} of the numerical flux and the definition of the projector $\PiLdue$, we have
	\begin{multline*}
		\Au(U,W - W^\pi) = 	\\
		\underbrace{(\imu \fu ,\fw - \fw^\pi)_{\Omega}}_{=0} \underbrace{- 
			(\fu  ,\nabla (w - w^\pi))_{\Th}
			+(\fu \cdot \n, w - w^\pi)_{\Eh}}_{=(\nabla\cdot\fu,w - w^\pi)_{\Th} = 0} + \tu h^{-1} (\PiE{k-1}(u - \hu), w - w^\pi)_{\Eh}  \\ 	\underbrace{- (u ,
			\nabla \cdot (\fw - \fw^\pi))_{\Th}\pm  (u , (\fw - \fw^\pi)\cdot \n)_{\Eh}}_{= (\nabla u , \fw - \fw^\pi) - (u,(\fw - \fw^\pi)\cdot \n)_{\Eh} = - (u,(\fw - \fw^\pi)\cdot \n)_{\Eh}} +  {(\hu , (\fw - \fw^\pi)\cdot \n)_{\Eh}} \\ 
		-\underbrace{( \fu  \cdot \n, \hw - \hw^\pi)_{\Eh} }_{=0} - \underbrace{\tu h^{-1} (\PiE{k-1} (u - \hu), \hw - \hw^\pi)_{\Eh}}_{=0}+ \underbrace{\Ku(\hu,\hw - \hw^\pi)_{\GR}}_{=0}\\ 
		= (\PiE{k-1}(u - \hu),  \tu h^{-1} \PiE{k-1} (w - w^\pi))_{\Eh} 
		-(u - \hu, (\fw - \fw^\pi)\cdot \n)_{\Eh},
	\end{multline*}
	where we could add the projector $\PiE{k-1}$ to the first term in the scalar product on the right hand side, as the second term therein belongs to $\Poly{k-1}(\Eh)$.	Moreover we have that
	\begin{multline*}
		\Au(W - W^\pi,U) = 	
		\underbrace{(\imu (\fw - \fw^\pi) ,\fu)_{\Omega}}_{=0} \underbrace{- 
			(\fw - \fw^\pi ,\nabla u)_{\Th}}_{=0} 
		+(
		(\fw - \fw^\pi)\cdot \n, u)_{\Eh}
		\\	+ \underbrace{\tu h^{-1} (\PiE{k-1}(w - w^\pi - \hw + \hw^\pi),u)_{\Eh} }_{= \tu  h^{-1} (\PiE{k-1}(w - w^\pi),u)_{\Eh}} 
		\underbrace{- (w - w^\pi ,
			\nabla \cdot \fu)_{\Th}
		}_{=0} +  \underbrace{(\hw - \hw^\pi, \fu \cdot \n)_{\Eh}}_{=0} \\
		- {( (\fw - \fw^\pi)  \cdot \n, \hu )_{\Eh} } - \underbrace{\tu h^{-1} (\PiE{k-1}(w - w^\pi - \hw + \hw^\pi), \hu)_{\Eh}}_{= \tu h^{-1} (\PiE{k-1}(w - w^\pi), \hu)} + \underbrace{\Ku(\hw - \hw^\pi,\hu)_{\GR}}_{=0} \\ 
		=(
		(\fw - \fw^\pi)\cdot \n,u - \hu)_{\Eh}  + \tu h^{-1}(\PiE{k-1} (w - w^\pi), \PiE{k-1} (u - \hu))_{\Eh},
	\end{multline*}
	which concludes the proof.
\end{proof}


\bibliographystyle{plain} 
\bibliography{biblio}

\end{document}